\documentclass[a4paper, 11pt]{amsart}
\usepackage[english]{babel} 
\usepackage{enumitem}

\usepackage[pagebackref]{hyperref} 
    \hypersetup{
        unicode=true,
        colorlinks,
        filecolor=black,
        citecolor=purple, 
        linkcolor=red}
        
\usepackage{longtable} 
\usepackage{amsmath} 
\usepackage{amssymb} 
\usepackage{amsthm} 
\usepackage{amsfonts}
\usepackage{mathrsfs} 
\usepackage{euscript} 
\usepackage{indentfirst} 
\usepackage{bm}

\usepackage{graphicx, float, color} 
\graphicspath{{./pictures/},{./tree_classes/}} 

\usepackage[mode=buildnew,subpreambles=true]{standalone}
\usepackage{tikz}
    \usetikzlibrary{arrows.meta}
    \usetikzlibrary{calc}
    \usetikzlibrary {decorations.pathmorphing}
    \tikzset{st/.style={shape=circle, draw, line width=1mm, inner sep=1mm, }, }

\numberwithin{equation}{section} 
\newcommand{\mmm}{\setminus}
\newcommand{\rr}{\mathbb{R}}
\newcommand{\nn}{\mathbb{N}}
\newcommand{\zz}{\mathbb{Z}}
\newcommand{\eS}{{\EuScript S}}

\newcommand{\eK}{{\EuScript K}}
\newcommand{\eP}{{\EuScript P}}
\newcommand{\al}{\alpha}
\newcommand{\be}{\beta}
\newcommand{\ga}{\gamma}
\newcommand{\Ga}{\Gamma}

\newcommand{\td}{\widetilde d}
\newcommand{\IN}{\subset}
\newcommand{\NI}{\supset}
\newcommand{\0}{\varnothing}
\newcommand{\8}{\infty}
\newcommand{\bj}{{\bf {j}}}
\newcommand{\bi}{{\bf {i}}}
\newcommand{\bk}{{\bf {k}}}

\renewcommand{\leq}{\leqslant}
\renewcommand{\geq}{\geqslant}
\newcommand{\dd}{{\partial}}

\newcommand{\bma}{{\bm{\alpha}}}
\newcommand{\bmb}{{\bm{\beta}}}

\newtheorem {theorem}               {Theorem}    [section]
\newtheorem {proposition} [theorem] {Proposition}
\newtheorem {lemma}       [theorem] {Lemma}
\newtheorem {corollary}   [theorem] {Corollary}
\newtheorem {definition}  [theorem]          {Definition} 
\newtheorem {remark}                {Remark}

\begin{document}
\title{On the classification of fractal square dendrites}


\author{Dmitry Drozdov}
\address{Novosibirsk State University, Novosibirsk, Russia}
\email{dimalek97@yandex.ru}


\author{Andrei Tetenov}
\address{Sobolev Institute of Mathematics,Novosibirsk, Russia},
\email{a.tetenov@gmail.com}



\begin{abstract}
 We consider the classification of fractal square dendrites $K$ based on the types of the self-similar boundary $\partial K$ and the main tree $\gamma$ of such dendrites.
We show that the self-similar boundary   of a  fractal square dendrite $K$  may be of 5 possible types and may consist of 3,4 or 6 points.
 We prove that the main trees of   fractal square dendrites  belong to 7 possible classes.
Bearing in mind the placement  and orders of the points of $\partial  K$ with respect to the main tree $\gamma$,  this results in 16 possible types of main trees for non-degenerate fractal square dendrites.\\

Keywords: fractal square, dendrite, self-similar boundary, main tree, ramification point\\

MSC {28A80}
\end{abstract}

\maketitle

\section{Introduction}

Fractal squares may be considered a special self-similar case of self-affine Bedford-McMullen carpets \cite{bedf1984, McMul}.
The later ones were defined in 1984 but the first examples of fractal squares such as Sierpinski gasket, Sierpinski carpet, and Viczek fractal appeared long before that time and are well known. 

Bedford-McMullen carpets and Sierpinski sponges were intensely studied in recent decades, and this study was concentrated mainly on dimension problems \cite{elekes2010self, fraser-bmc, Olsen1998, peres1994}.
Fractal squares are simpler and more convenient. They are  the self-similar sets that satisfy the open set condition. 
The Hausdorff dimension of a fractal square is equal to its similarity dimension and box dimension, and each fractal square has finite positive measure.

Since 2008, L.~L.~Cristea and B.~Steinsky \cite{cristea2009, cristea2011, cristea2017, cristea2020} studied labyrinth fractals, a class of fractal squares that were dendrites. 
Moreover, this class admitted a composition operation, which allowed to consider random labyrinth fractals.

Topological structure of fractal squares was studied by K.-S.~Lau, J.~J.~Luo and H.~Rao  \cite{LLR2013}.
They subdivided the fractal squares into $3$ types, depending on the topology of the complement $H=\rr^2\mmm (\zz^2+K)$ of a given fractal square $K\IN [0,1]^2$ in the plane $\rr^2$. 
As it follows from this work, if $H$ has a unique unbounded connected component, then $K$ is totally disconnected, if $H$ contains infinitely many unbounded components, then all components of $K$ are parallel line segments or one-point components, and if $H$ has no unbounded components, then $K$ has a connected component, different from a point or a line segment.

Recently Xiao \cite{Xiao2021} gave a complete characterization  of fractal squares that have a finite number of connected components.

The central result of the paper is the classification of fractal square dendrites by the type of their main trees (Theorem \ref{thm:7trees}). Each fractal square dendrite $K$ is post-critically finite and therefore it has a finite self-similar boundary $\dd K$. The minimal subdendrite $\ga\IN K$, which contains 
$\dd K$ is called the main tree of the dendrite $K$ (Definition \ref{mtree}). Our Theorem \ref{thm:7trees} states that the main trees of fractal square dendrites belong to 7 types.

In Section 3 we consider the intersections of pieces of a fractal square and find the finiteness conditions for  such intersections. In Section 4 we prove that there are 5 boundary types for fractal square dendrites. In Section 5 we find which combinations of pieces are forbidden and which ramification orders are permitted for each boundary type, which yields the main result in Section 6.

\section{ Preliminaries}

\subsection{ Self-similar sets\\}

\begin{definition}
Let $\eS=\{S_1, \ldots, S_m\}$ be a system of  contraction similarities in  $\mathbb{R}^n$.  
A non-empty compact set $K$ satisfying the equation  $K=\bigcup\limits_{i=1}^m S_i(K)$ is called {\em the attractor of the system} $\eS$ or the set that is self-similar with respect to the system $\eS$.
\end{definition}

We denote by  $I=\{1,\ \ldots,\ m\}$  the set of indices of the system $\eS$, then $I^{*}=\bigcup\limits_{n=1}^\infty I^n$ denotes the set of all words $\bm{i}=i_1\ldots i_n$ of finite length in the alphabet $I$, called {\em multi-indices}.
We will use the notation $S_{\bj}=S_{j_{1}j_{2}\ldots j_{n}}=S_{j_{1}}S_{j_{2}} \ldots S_{j_{n}}$, and  denote the $S_{\bj}(K)$ by $K_{\bj}$.
The sets $K_\bj$,  where $\bj=j_{1}j_{2}\ldots j_{n}$,  are called {\em the pieces of  order $n$} of the set $K$.

The set of all  infinite strings (or addresses)  $I^\infty=\left\{\bm\beta = \beta_1\beta_2\ldots, \; \beta_i\in{I} \right\}$ is called {\em the index space} of the system $\eS$.
The mapping $\pi:I^\8\to K$ that maps a sequence $\bm\beta\in I^\8$ to the point $x=\bigcup\limits_{n=1}^\8 K_{\beta_1\ldots\beta_n}$ is called the {\em address map} for the attractor $K$.
Then for each $x\in K$, the set $\pi^{-1}(x)$ is a set of addresses of the point $x$.

\begin{definition}{\rm \cite[Def. 1.3.4]{kig}}\label{def:bd} 
{\em  A critical set} of the attractor $K$ of the system $\eS$ is the set $C=\{x:\; x\in S_i(K)\cap S_j(K),\; S_i, S_j\in\eS\}$ of points of pairwise intersections of pieces $S_i(K)$ of $K$.
The set $\dd K$ of all $x\in K$ such that for some $\bj\in I^*$, $S_\bj(x)\in C$ is called the {\em self-similar boundary} of the set $K$.
Its complement in $K$ we denote by $\dot K$.
\end{definition}

\begin{definition}
We denote by $\tilde{\Gamma}(\eS)$ the  {\em ordinary  intersection graph} for $K(\eS)$  which is is a graph with a vertex set $\{K_i:\; i\in I\}$, and a pair of vertices $K_i,K_j$  is connected by an edge if $K_i\cap K_j\neq\0$.
\end{definition}

\begin{theorem} [\cite{hata85,kig}]
The attractor $K$ of a system $\eS$ is connected iff its ordinary intersection graph $\tilde{\Gamma}(\eS)$ is connected.\\
If $K$ is connected, it is locally connected.
\end{theorem}

Throughout this paper we consider the  case when the maps $S_i$ are the similarities  and the attractor $K$ is connected.  
In this case we say $K$ is a {\em self-similar continuum}. 

We  say that the system $\eS$ {\em satisfies one-point intersection property} if for any non-equal pieces   $K_i, K_j,\; i,j\in I$ of order 1 of the attractor $K$, $\#(K_i\cap K_j)\le 1$.
We define the (bipartite) intersection graph for such systems:

\begin{definition}[\cite{fip}]
Let a set $K=K(\eS)$ be a  self-similar continuum which possesses  one-point intersection property.
{\em The intersection graph} $\Gamma(\eS)$ of the system $\eS$ is a bipartite graph with parts $\eK=\{K_i:\; i\in I\}$ and $\eP=\{p:\;p\in K_i\cap K_j,\; i,j\in I, i\neq j\}$, and with a set of edges $E=\{(K_i,p):p\in K_i\}$.
\end{definition}

We call $K_i\in\eK$ {\em white vertices} 
 and $p\in \eP$ {\em black vertices} of the graph $\Ga$.

\begin{definition}[\cite{MW1988}]\label{def:gds}
Let $\{\eS_{ij}\mid i,j=1,\ldots,m\}$ be a finite set of systems of contractive maps.
Let $T_{ij}$ be the operators defined by $T_{ij}(A)=\bigcup\limits_{S_{ijk}\in\eS_{ij}}S_{ijk}(A)$.
Then the family of compact sets
$$\left\{K_j=\bigcup\limits_{i=1}^m T_{ij}(K_i)\mid j=1,\ldots,m\right\}$$ 
is called the attractor of the graph-directed system $\{\eS_{ij}\mid i,j=1,\ldots,m\}$.
\end{definition}



\subsection{Dendrites\\}

\begin{definition}
A {\em dendrite} is a locally connected continuum containing no simple closed curve.
A {\em self-similar dendrite} is a self-similar continuum, which is a dendrite.
\end{definition}

We shall use the notion of {\em order of a point} in the sense of Menger-Urysohn (see \cite[Vol.2, \S 51, p.274]{Kur}) and we denote by $Ord (p, X)$ the order of the continuum $X$ at a point $p\in X$. 
Points of order $1$ in a continuum $X$ are called {\em end points} of $X$; the set of all end points of $X$  will be  denoted by $EP(X)$. 
A point $p$ of a continuum $X$ is called a {\em cut point} of $X$ provided that $X \setminus \{p\}$ is not connected; the set of all cut points of $X$ will be denoted by $CP(X)$. 
Points of order at least $3$ are called {\em ramification points} of $X$; the  set of all ramification points of $X$ will be denoted by $RP(X)$.

We  use the following statements selected from  \cite[Theorem 1.1]{Char}:
\begin{theorem} For a continuum $X$ the following conditions are equivalent:
\begin{enumerate}[nolistsep]
    \item $X$ is dendrite;
    \item every two distinct points of $X$ are separated by a third point;
    \item each point of $X$ is either a cut point or an end point of $X$;
    \item each non-degenerate subcontinuum of $X$ contains uncountable many cut points of $X$.
    \item for each point $p \in X$ the number of components of the set $X \setminus \{p\} = ord (p, X)$ whenever either of these is finite;
    \item the intersection of every two connected subsets of X is connected;
    \item $X$ is locally connected and uniquely arcwise connected.
\end{enumerate}
\end{theorem}

\begin{definition}\label{gaxy}
For a finite subset $A$ of a dendrite $X$, we denote by $\ga(A)$ the smallest subdendrite of $X$ containing $A$. Particularly $\ga(x,y)$ denotes
the unique subarc  in $X$ with endpoints $x,y$. 
\end{definition}

\begin{theorem}[\cite{fip, sss2}]\label{thm:den_sufficient}

Let a set $K=K(\eS)$ be a  self-similar continuum which possesses  one-point intersection property.
If the intersection graph $\Gamma(\eS)$ of the system $\eS$ is a tree, then its attractor $K$ is a dendrite.
\end{theorem}

\begin{definition}\label{mtree}
Let $K$ be a { self-similar dendrite} with a finite self-similar boundary $\dd K$. 
The minimal subdendrite $\ga\IN K$ containing $\dd K$ is called {\em the main tree} of the self-similar dendrite $K$.
\end{definition}

\subsection{Fractal squares\\}

\begin{definition}
Let $D=\{d_1,\ldots,d_m\}\IN\{0,1,\ldots,n-1\}^2$, where $n\ge 2$, and $1<m<n^2$.
{\em A fractal square} of  order $n$  with a {\em digit set $D$} is a compact set $K\IN\rr^2$,  satisfying the equation
\begin{equation}\label{fsqeq}
 K=\dfrac{K+D}{n}   
\end{equation}
\end{definition}
\noindent The equation \eqref{fsqeq} will be  often used in its equivalent form $nK=K+D$.
Since the unit square $P=[0,1]^2$ satisfies the equation $nP\NI P+D$ the set $K$ is contained in $P$.
\begin{figure}[H]
    \centering
    \includegraphics[width=0.35\textwidth]{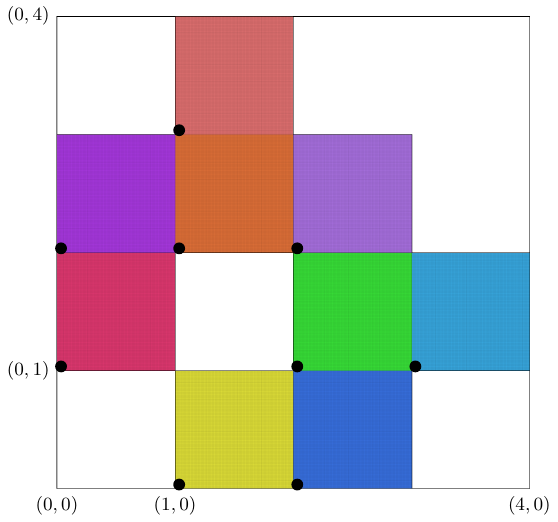}
    \hspace{1cm}
    \includegraphics[width=0.3\textwidth]{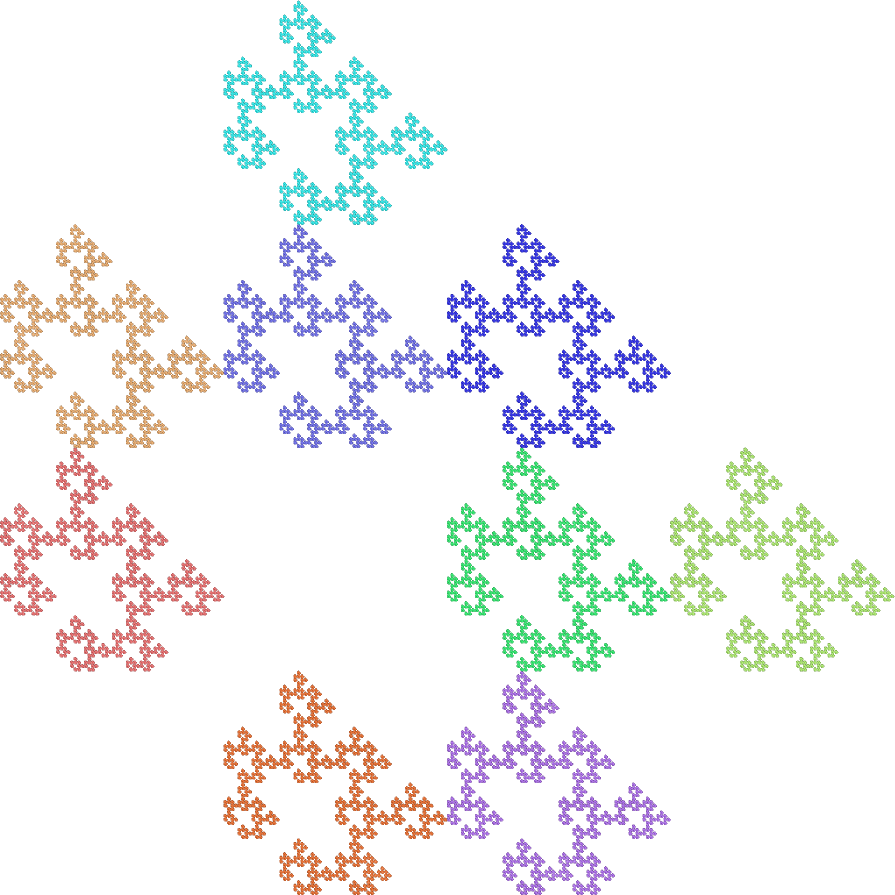}
    \caption{A fractal square (right) and the  set $D+P$ (left) The elements of $D$ are marked by black points.}
    \label{fig:fr_sq}
\end{figure}
The equation \eqref{fsqeq} defines a system $\eS$ of homotheties $S_i(x)=\dfrac{x+d_i}{n}$, 
where $i\in I=\{1,...,m\}$,
and the Hutchinson's operator $T_\eS$ of the system $\eS$ is defined by the equation $T_\eS(A)=\dfrac{D+A}{n}$. \\

\begin{definition}\label{refin}
For any $k\in \nn$, {\em the $k$-th refinement} of the system $\eS$ is a system $\eS^k=\{S_\bi, \bi\in I^k\}$, where  $S_\bi(x)=\dfrac{x+d_\bi}{n^k}$ and $d_\bi=n^{k-1}d_{i_1}+n^{k-2}d_{i_2}+...+d_{i_k}$. 
\end{definition}

The system $\eS^k$ represents $K$ as a fractal cube of order $n^k$ with the digit set 
$D^k=n^{k-1}D+n^{k-2}D+...+D$ and pieces $\dfrac{K+d_\bi}{n^k}$.
The Hutchinson  operator $T_\eS^k$ of the system $\eS^k$ is defined by the equation
$T_\eS^k(A)=\dfrac{D^k+A}{n^k}$.

Each infinite string 
$\bi=i_1i_2\ldots\in I^\8$ defines a unique point $x=\pi(\bi)$ by the formula
$\pi(\bi)=\sum\limits_{k=1}^n \dfrac{d_{i_k}}{n^k}$.

\begin{remark} \label{rmk:fsd}
From now on, unless otherwise specified, by the words  {\em"fractal square $K$"} we mean {\em " fractal square $K$ of order $n$ with a digit set $D$"}. 
If a fractal square is a dendrite, we call it a "fractal square dendrite" and use abbreviation "FSD" which usually would mean {\em "a fractal square dendrite of order $n$ with a digit set $D$"}. 
\end{remark}

\section{The intersections of the pieces of a fractal square}

In  Subsection {\it 3.1} we consider the face intersection sets $F_\bma$ (Definition \ref{def-falpha}) and show that the pieces of a fractal square intersect by homothetic images of these sets $F_\bma$ (Proposition \ref{fbma}) and find the equations for these sets (Theorem \ref{thm_falpha}).
In Subsection {\it 3.2} we obtain the conditions under which the side intersection sets $F_\bma$ are singletons
(Corollary \ref{onepoint}).

\subsection{The faces $K_\al$ and the face intersection sets $F_\al$ of a fractal square\\}

\begin{definition}\label{setA}
We denote by $A$ the set $\{-1,0,1\}^2$ of  2-tuples $\bma=(\al_1,\al_2)$.    
\end{definition}

There is a one-to-one correspondence between the set $A$ and the set of faces of the unit square $P=[0,1]^2$.
\begin{definition}
Each $\bma\in A=\{-1,0,1\}^2$, defines a unique  face $P_\bma$ of the unit square $P$ by the equality $P_\bma=P\cap(P+\bma)$.
\end{definition}
The symmetry of opposite faces of $P$ is expressed by the equalities $$P_\bma=P\cap(P+\bma)=((P-\bma)\cap P)+\bma=P_{-\bma}+\bma$$

There is an order relation $\sqsubseteq$ on the set $A$ which is  order isomorphic to the inclusion ordering $\supseteq$ of the faces of $P$.
\begin{definition}\label{Aorder}
 Given $\bma=(\al_1\al_2)$, $\bmb=(\be_1\be_2)\in A$, we write $\bma\sqsubseteq \bmb $ if $\al_i\neq 0$ implies $\be_i=\al_i$, and 
$\al_i= 0$ implies $\be_i\neq 0$.   
\end{definition}

Obviously,  $\bma\sqsubseteq\bmb$ iff $P_\bma\supseteq P_\bmb$. The maximal elements of $A$ with respect to the  relation $\sqsubseteq$ are $(\pm 1,\pm 1)$, so the faces $P_{(-1,-1)}, P_{(1,-1)}, P_{(-1,1)}, P_{(1,1)}$ correspond to the respective corner points $(0,0), (1,0), (0,1), (1,1)$ of $P$.

\begin{definition}\label{def-falpha}
The face $K_\bma$ of a fractal square $K$ is the intersection $K\cap P_\bma$.\\ We call the intersection of opposite faces $K_\bma\cap(K_{-\bma}+\bma)$ the {\em face intersection set} $F_\bma$.
\end{definition}

\begin{figure}[H]
    \centering
    \includegraphics[width=0.3\textwidth]{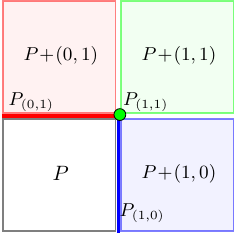}
    \hspace{1cm}
    \includegraphics[width=0.3\textwidth]{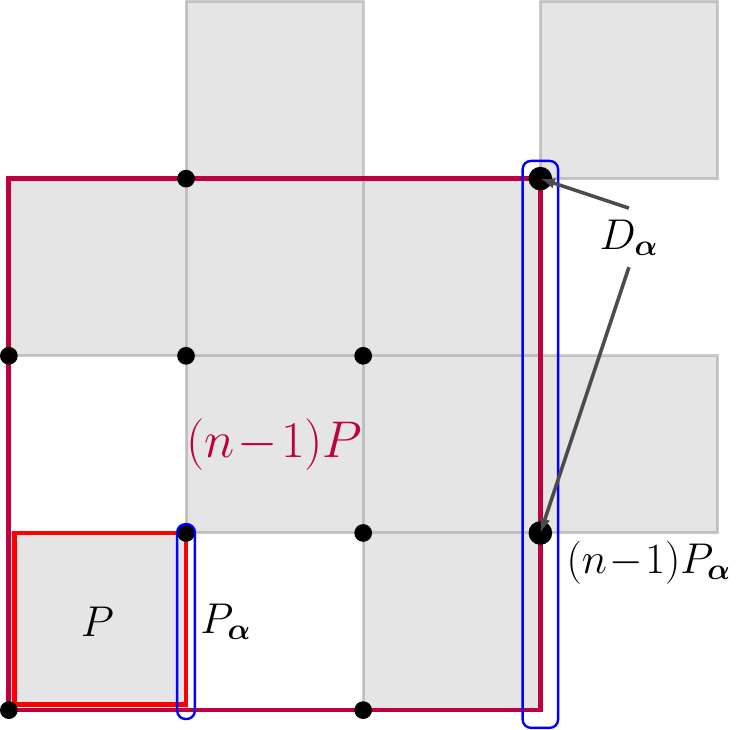}
    \caption{The sets $P_\bma$ and $D_\bma.$}
    \label{fig:faces}
\end{figure}

\begin{proposition}
The face $K_\bma$ of a fractal square $K$ satisfies the equation $n K_\bma=K_\bma+D_\bma$, where $D_\bma=D\cap(n-1)P_\bma$.
\end{proposition}

\begin{proof}
Note that $n(K\cap P_\bma)=(K+D)\cap nP_\bma.$ For any $d_i\in D$, $(d_i+K)\cap nP_\bma$ is empty unless $d_i\in(n-1)P_\bma.$ 
Therefore $d_i\in D\cap (n-1)P_\bma$. At the same time, $(d_i+K)\cap nP_\bma=d_i+K_\bma$, so $n K_\bma={K_\bma+D_\bma}$.
\end{proof}

If we project $K_\bma$ to the coordinate axis orthogonal to $\bma$, we obtain a fractal 1-cube or a {\em fractal segment} with a digit set $pr_\bma D_\bma$.

\begin{proposition}\label{fbma}
Let $K$ be a fractal square.\\
1. For any $\bma\in A$, $K\cap(K+\bma)=F_\bma=F_{-\bma}+\bma $;\\
2. If   $\bi,\bj \in I^k$, $\bi\neq \bj$ and  $K_\bi\cap K_\bj\neq\0$, then  $d_\bi-d_\bj=\bma$ for some $\bma\in A$, and   $$K_\bi\cap K_\bj=\dfrac{d_\bi+F_\bma}{n^k}$$
\end{proposition}

\begin{proof}
1.  Since $P\cap P+\bma=P_\bma=P_{-\bma}+\bma$, we have the chain of equalities $$K\cap(K+\bma)= K_\bma\cap (K_{-\bma}+\bma)=F_\bma=F_{-\bma}+\bma.$$

2. Note that $S_\bi^{-1}(K_\bi\cap K_\bj)=K\cap S_\bi^{-1}(K_\bj)=K\cap (K+d_\bj-d_\bi)$. 
Since by Definition \ref{refin}, $d_\bj-d_\bi\in \zz^k$, the last two squares are adjacent only if $d_\bj-d_\bi$  is equal to some  $\bma\in A$. 
In this case, $K\cap(K+d_\bj-d_\bi)=F_\bma$.
Therefore $K_\bi\cap K_\bj= S_\bi(F_\bma)= \dfrac{d_\bi+F_\bma}{n^k}$. 
\end{proof}

From the Proposition \ref{fbma} one sees that
for any fractal square there are only 4 possible types of intersections of its adjacent pieces of the same size. 
Any of these intersections is the image of $F_{10}, F_{01}, F_{11}$ or  $ F_{-1,1}$ under some map $S_\bi$. 
The  sets $F_{11}$ and  $ F_{-1,1}$ correspond to the intersections of opposite corners and therefore are either singletons or empty sets. 

For any $i,j\in\{-1,1\}$ the set $F_{ij}$ is a singleton $\{(i,j)\}$ if\\   $(n-1)\left(\frac{i+1}{2},\frac{j+1}{2}\right)\in D_1$ and $(n-1)\left(\frac{-i+1}{2},\frac{-j+1}{2}\right)\in D_2$, otherwise it is empty. 

The sets $F_{10}, F_{01}$ correspond to vertical and horizontal adjacency of pieces. 
These sets  are the intersections of opposite sides of $K$, which are fractal segments.

The equations for the sets $F_\bma$  are given by the following theorem.

\begin{theorem}\label{thm_falpha}
For $\bma\in A\mmm\{0\}$ the set $F_\bma$ satisfies the equation
\begin{equation}\label{sideint}
 F_\bma=\bigcup\limits_{\bmb\sqsupseteq\bma} \dfrac{F_\bmb+G_{\bma\bmb}}{n}
 \end{equation}
where 
$G_{\bma\bmb}=D_\bma\cap(D_{-\bma}+n\bma-\bmb)$.
\end{theorem}

\begin{proof} By Definition \ref{def-falpha},
\begin{equation}\label{line}
    nF_\bma=nK_\bma \cap (nK_{-\bma}+n\bma)= 
    \bigcup\limits_{d_1,d_2\in D}(K+d_1)\cap (K+d_2+n\bma)
    \end{equation}
$(K+d_1)\cap (K+d_2+n\bma)\neq\0$ implies that $d_2-d_1+n\bma=\bmb$  for some $\bmb\in A$.

Taking $i$-th coordinate for all entries of the last equality, we obtain that
$$d_{2i}-d_{1i}+n\al_i=\be_i$$ 
The relations $\al_i,\be_i\in\{-1,0,1\}$ and $|d_{2i}-d_{1i}|\le n-1$ imply  that:\\ 
1. if $\al_i=1$ then $\be_i>0$, therefore $\be_i=\al_i$ and $d_{2i}-d_{1i}=n-1$;\\
2.  if $\al_i=-1$ then $\be_i<0$, therefore $\be_i=\al_i$ and $d_{2i}-d_{1i}=1-n$;\\
3. if $\al_i=0$ then $d_{2i}-d_{1i}=\be_i$ which may be $-1,0$ or $+1$. \\
By  Definition \ref{Aorder}, $\bmb\sqsupseteq\bma$, and there are three possible choices of $\bmb$ if $\{\al_1,\al_2\}\ni 0$.\\
From the other hand, the equality $d_1=d_2+n\bma-\bmb$ shows that $d_1$ belongs to the set $ D\cap(D+n\bma-\bmb)$, which we denote by $G_{\bma\bmb}$.\\
For each $d_1\in G_{\bma\bmb}$, $(K+d_1)\cap (K+d_2+n\bma)= (K\cap(K+\bmb))+d_1= F_\bmb+d_1$.
The last equality shows  that $nF_\bma= \bigcup\limits_{\bmb\sqsupseteq\bma} {F_\bmb+G_{\bma\bmb}}$.
\end{proof}

\subsection{The cardinality of intersections of  fractal segments\\}

An  argument similar to that of Theorem \ref{thm_falpha} gives  the intersection formula for fractal segments:

\begin{proposition}
Let  $K_1=\dfrac{K_1+D_1}{n}$ and $K_2=\dfrac{K_2+D_2}{n}$ be fractal segments. 
The set $F_0 =K_{1}\cap K_{2}$ satisfies the equation   
\begin{equation*}
F_0=\dfrac{F_0+G_{0}}{n}\cup\dfrac{F_{-1}+G_{-1}}{n}\cup\dfrac{F_{1}+G_{1}}{n}, 
\end{equation*}
where for any $a\in\{-1,0,1\}$,  $G_{a}=D_1\cap(D_2-a)$ and $F_a=K_1\cap (K_2+a)$.\\
If $F_{1}\neq\0$ then $0\in D_2$ and $(n-1)\in D_1$, and  $F_1=\{1\}$.\\ 
If $F_{-1}\neq\0$ then  $0\in D_1$ and $(n-1)\in D_2$, and  $F_{-1}=\{0\}$. \hfill$\square$
 \end{proposition}

Let $K_1$ and $K_2$ be fractal segments of order $n$ with digit sets $D_1$ and $D_2$ such that $0\in D_1$ and $n-1\in D_2$ and therefore $0\in K_1$ and $1\in K_2$.
If $k\in D_1$ and $k-1\in D_2$, then $k/n\in K_1\cap K_2$ and this point is called a {\em junction point} of $K_1$ and $K_2$.

\begin{corollary}\label{fin_int}
Let $K_1=\dfrac{K_1+D_1}{n}$ and $K_2=\dfrac{K_2+D_2}{n}$ be fractal segments.
\begin{itemize}[nolistsep]
    \item[(i)] If $\#G_0> 1$, then the set $F_0$ is uncountable.
    \item[(ii)] If $\#G_0\le 1$, then only one of the sets $F_{-1},F_1$ is non-empty and $F_0$ is countable.
    \item[(iii)] The set $F_0$ is finite in the following cases:
    \begin{itemize}[nolistsep]
        \item[\textbf{(a)}] $\#G_0=1$ and $(F_{-1}+G_{-1})\cup(F_{1}+G_{1})=\0$;
        \item[\textbf{(b)}] $\#G_0=0$ and  $(F_{-1}+G_{-1})\cup(F_{1}+G_{1})\neq\0$.
    \end{itemize}
\end{itemize} 
\end{corollary}

\begin{proof}
Since for $i\in\{-1,+1\}$,$\#F_{i}\le 1$  the sets $\dfrac{F_{i}+G_{i}}{n}$ are finite.

(i) Note that if $\#G_0>1$, then $F_0$ contains the set $Q_0$ defined by the equation
$n Q_0=Q_0+G_0$, which is uncountable.\\

(ii) Suppose $\#G_0=1$, i.e. $G_0=\{d_0\}$, and $(F_1+G_1)\cup(F_{-1}+G_{-1})\neq\0$. \vspace{1mm}\\
Let $\td_0=\dfrac{d_0}{n-1}$ denote the fixed point of the homothety $S_{d_0}(x)=\dfrac{x+d_0}{n}$.
For any $x_0\in  \dfrac{(F_1+G_{1})\cup(F_{-1}+G_{-1})}{n} $ the set $F_0$ contains a sequence $x_k=\dfrac{x_0-\td_0}{n^k}+\td_0$.
Therefore $F_0$ is an infinite countable set.\\

(iii) In the case \textbf{(a)} we get $F_0=\{\td_0\},$ therefore $\#F_0=1$.\\

In the case \textbf{(b)} $n F_0=(F_1+G_1)\cup (F_{-1}+G_{-1})$, thus $\#F_0=\#F_{1}\cdot\#G_1+\#F_{-1}\cdot\#G_{-1}<\8$.
\end{proof}

\begin{corollary}\label{onepoint} 
The set $F_0$ is a singleton if \\
\textbf{(a)} $\#G_0=1$ and $\#F_{1}\cdot\#G_1+\#F_{-1}\cdot\#G_{-1}=0$; or\\
\textbf{(b)} $G_0=\0$ and 
$\#F_{1}\cdot\#G_1+\#F_{-1}\cdot\#G_{-1}=1$. \hfill $\square$
\end{corollary}



\section{The self-similar boundary and the main tree of a fractal square dendrite}

\begin{proposition}
\label{thm:den_necessary_sufficient}
If a fractal square $K$ is a dendrite then it possesses one-point intersection property.
\end{proposition}

\begin{figure}[H]
    \centering
    \includegraphics[width=0.65\textwidth]{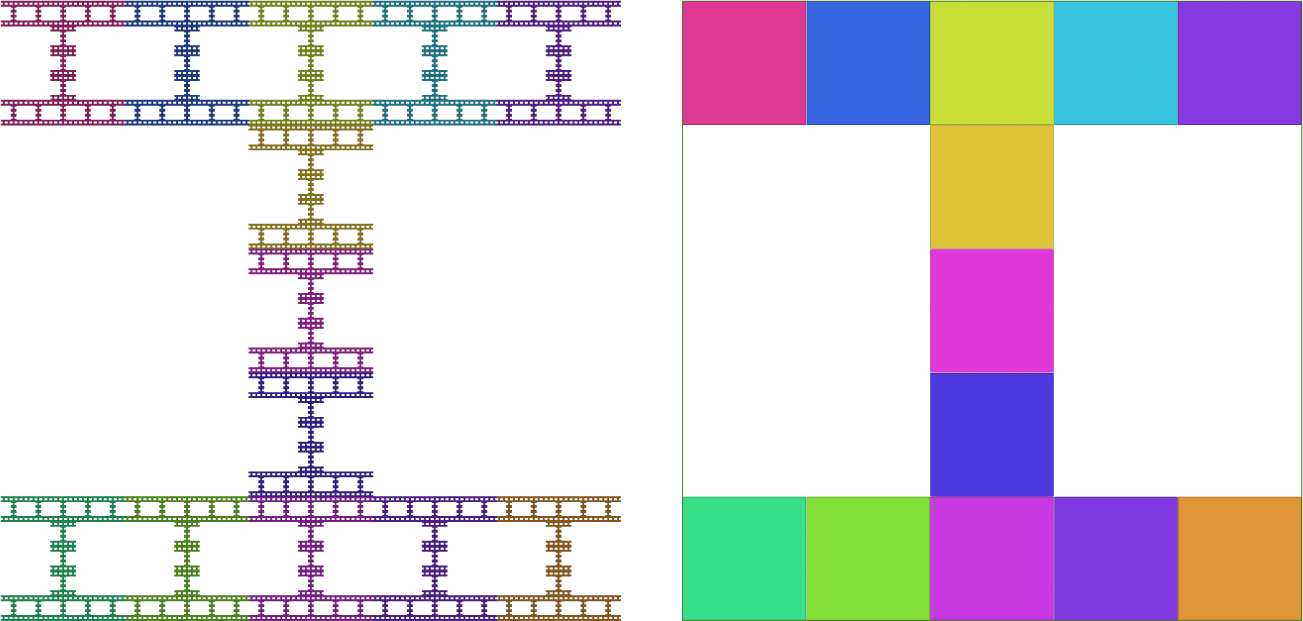}
    \caption{A fractal square  with a full side intersection.}
    \label{fig:line_int}
\end{figure}

\begin{proof}
The intersection $K_\bi\cap K_\bj$ of two pieces of a fractal square dendrite is simple-connected, therefore it is either a point or a straight line segment.The last is possible when $K_\bi\cap K_\bj$ is a full side, because
 a fractal segment $L$  contains a segment if and only if $L=[0,1]$. 

Suppose $K$ is a connected fractal square which has a  full side intersection of its pieces. Then, up to rotation, $K$ contains $[0,1]\times\{0,1\}$.
An example is shown on the Figure \ref{fig:line_int}.

In this case the set $D$ contains a subset $ \{0,...,n-1\}\times \{0,n-1\}$, therefore $K$ contains a subset $L=[0,1]\times\{0,1/n\}$. 
For each of the sets $K_{0k}=\dfrac{K+(0,k)}{n}$, the intersection $K_{0k}\cap L$ is a pair of parallel segments. 
By connectedness of $K_{0k}$ there is a path $\ga_{0k}\IN K_{0k}$ connecting these segments.   

The union $L\cup\ga_{00}\cap\ga_{n-1,0}$ contains a closed curve, hence $K$ is not a dendrite. 
\end{proof}

\begin{corollary}
A fractal square $K$ is a dendrite if and only if its bipartite intersection graph is a tree.\hfill$\square$  
\end{corollary}

Consider the set  $A=\{-1,0,1\}^2$ of all faces of a square. 
For a fractal square $K$ with a digit set $D$ we define the set 
$$A_D=\{\bma\in A\mmm\{(0,0)\}: D\cap(D+\bma)\neq\0\}.$$ 
The set $A_D$ is the set of those $\bma\in A$  for which there are the pieces $K_i, K_j$ such that $d_j-d_i=\bma$ and therefore  having common boundary points. 
For such pieces, their common boundary is $S_i(F_\bma)=S_j(F_{-\bma})$. 
The following statement is obvious:\\

\begin{proposition}
The self-similar boundary $\dd K$ of a fractal square $K$ is the union
\begin{equation}
 \dd K=\bigcup\limits_{\bma\in A_D}F_\bma.   
\end{equation}
\end{proposition}

Note that for any piece $K_\bi$ of $K$, the topological boundary of $K_\bi$ in $K$ is a subset of 
$S_\bi(\dd K)$. 
Thus, if a path $\ga\IN K$ has endpoints $a\in K_\bi$ and $b\in K\mmm K_\bi$, then $\ga$ contains a point of $S_\bi(\dd K)$.

\begin{theorem}\label{ssboundary}
Let $K$ be a FSD which is not a line segment. 
There are  the following types of the self-similar boundary $\dd K$:\\
If  $\dd K=F_\bma\cup F_{-\bma}\cup F_\bmb\cup F_{-\bmb}$ then $\dd K$ is of\\
type {\bf A} if $\bma=(1,0)$, $ \bmb=(0,1)$;\\ 
type {\bf B} if $\bma=(1,1)$, $ \bmb=(1,-1)$; \\
type {\bf C} if  $\bma=(1,1) \mbox{ or } (1,-1),\bmb=(1,0) \mbox{ or  }(0,1)$.\\
For type {\bf D}, $\dd K=F_{(1,0)}\cup F_{(-1,0)}\cup F_{(0,1)}\cup F_{(0,-1)}\cup F_\bmb\cup F_{-\bmb}$,\\ where
 $\bmb=(\be_1,\be_2)=(1,1) \mbox{ or  }(1,-1)$.
 
For the types {\bf A, B, C} the set $\dd K$ consists of 4 points. The type {\bf D} boundary  $\dd K$  consists of 6 points {\bf D6} or, in degenerate case, of 3 points {\bf D3}.
\end{theorem}

\noindent The degenerate case {\bf D3} arises if there is $\bma\in\{(1,0),(-1,0)\}$ and $\bmb\in\{(0,1),(0,-1)\}$ such that $F_\bma=F_\bmb$.

\begin{proof}
If  for some $\bma\in A$ the self-similar boundary  $\dd K=F_\bma\cup F_{-\bma}$  then $K$ is a line segment with endpoints in $F_\bma$ and $ F_{-\bma}$.
 
If the set $A_D$ contains $(1,1),(1,-1)$ and $\bma=(0,1)$ or $\bma=(1,0)$, then $\#G_\bma\geq 2$, so by
Corollary \ref{fin_int}, $F_\bma$ is uncountable, which is impossible if $K$ is a dendrite, therefore $\bma\notin A_D$. 
All other possibilities are enumerated for the types {\bf A, B, C, D}.

For the type {\bf A}, the set $\dd K$ can consist only of four points. 
If one of these points is a corner point, then $\dd K$ contains exactly two corner points that are the endpoints of some side of $P$ and three points of $\dd K$ lie on this side.

Suppose  $\dd K$ contains 3 corner points, say, $(1,0), (0,0)$ and $(0,1)$.
It implies that $F_{(1,-1)}\neq\0$.
If so, the set $\dfrac{D+P}{n}$ is connected and contains the points (1,0) and (0,1). 
We assert that in this case $G_{(-1,1)}=(D\cap (D+(1,-1))\neq\0$.
 
Indeed, let $D'$ be a subset of $D$ that satisfies the following conditions:\\
1. The set $D'+P$ is connected;\\
2. $D'\cap (D'+(1,-1))=\varnothing$;\\
3. $(0,n-1)\in D'$. 

Then $D'$ is equal  either to $\{0\}\times\{k,...,n-1\}$  or to $\{k,...,n-1\}\times \{n-1\}$ for some $0\leq k\leq n-1$.
This observation proves our assertion.
 
So  $F_{(1,-1)}\neq\0$, $(1,-1)\in A_D$ means that we have the case {\bf D}. 
 
The same argument shows that in the case {\bf C} the set $\dd K$ can consist only of four points. 
Case {\bf B} is obvious.

For the type {\bf D} the set $\dd K=F_{(1,0)}\cup F_{(-1,0)}\cup F_{(0,1)}\cup F_{(0,-1)}\cup F_{(1,1)}\cup F_{(-1,-1)}$.
Let  $\bma\in\{(1,0),(-1,0)\}$ and $\bmb\in\{(0,1),(0,-1)\}$. 
Then the equalities  $F_\bma=F_{-\bmb}$, $F_{\bmb}=F_{\bma+\bmb}$ and $F_{-\bma}=F_{-\bma-\bmb}$ are equivalent and $\dd K=F_{-\bma}\cup F_\bma\cup F_\bmb$. 

This shows that  for the type {\bf D} the inequality $\#\dd K<6$ implies $\#\dd K=3$. 
We mark these two subtypes by {\bf D3} and {\bf D6}.
\end{proof}

As follows from Theorem \ref{ssboundary}, the self-similar boundary of a FSD  consists of 3, 4 or 6 points.

\begin{definition}\label{def:mtree}
Let $K$  be a FSD possessing finite self-similar boundary $\dd K$. 
The minimal subdendrite $\hat\ga\IN K$, containing  $\dd K$ is called the {\em main tree} of the dendrite $K$.
\end{definition}

We provide some statements concerning the main tree $\hat\ga$ of a self-similar dendrite $K$ with a finite self-similar boundary; see \cite{polden}.\\
1. If $x\in \hat\ga$ and for any component $Q$ of $K\mmm\{x\}$, $Q\cap\dd K\neq\0$, then $Ord(x,K)=Ord(x,\hat\ga)$.\\
2. If $x$ is a cut point  of dendrite $K$, then either $x\in S_\bi(\dd K)$ for some  $\bi$, or for any component $Q$ of $K\mmm\{x\}$, $Q\cap\dd K\neq\0$.\\
3. The set $CP(K)$ of cut points of $K$ is the union of all sets $S_\bi(\hat\ga), \bi\in I^*$. 
Therefore, $\dim_H(CP(K))=\dim_H(\hat\ga)$.

\section{The order of ramification points  of a fractal square dendrite}

\subsection{Forbidden sets of digits for different types of fractal square dendrites\\}

For each type of FSD there are combinations of digits which cannot occur in $D^k$ for any $k$. 
The following Lemma considers such combinations. 
We say that a set $Q\IN \zz^2$ {\em  is forbidden} (for the given type of $K$) if for any $k\in\nn$ and $\td\in D^k$,
$\td+Q\not\IN D^k.$

\begin{lemma}\label{quadruples}
The following  digit combinations $Q$ are forbidden:\\
a)  $\{(0,0), (1,0)\} $ and $\{(0,0), (0,1)\} $ for the type {\bf B} ;\\
b)  $\{(0,0), (0,1), (1,0), (1,1)\}$ for all types except {\bf C};\\
c)  $\{(0,0),\bma,\bmb,\bma+\bmb\}$ for the type {\bf C};\\
d3) the triple $\{a_1,a_2, a_3\}$ for the type {\bf D3}, if $\dd K=\{a_1,a_2, a_3\}$;\\  
d6) the triples $\{(0,0),(\be_1,0),\bmb)\}$,  $\{(0,0),(0,\be_2),\bmb)\}$ for the type {\bf D6}. 
\end{lemma}

\begin{proof} 
The case a)  is obvious because for any $\bma\in \{(0,1),(1,0)\}$ the set $ F_\bma$ is uncountable.

To prove that a set $Q$ of integer vectors is forbidden  for other types we  verify that a set $Q+K$ contains non-contractible closed loops.

(b) For the types {\bf A, D6} we suppose that $\bma=(1,0), \bmb=(0,1)$,
$F_{-\bma}=\{(0,b)\}, F_{-\bmb}=\{(a,0)\}$ where $a,b$ belong to the interval $]0,1[$. Since one of the sets $F_{(1,1)}$ or $F_{(1,-1)}$ is empty,
there is an open square
$P_1=\dfrac{x+\dot P}{n}$, $x\in \{(-1,-1),(0,-1), (-1,0), (0,0)\}$ which has an empty intersection with $Q+K$. Denote $x_1=(a,0)$, 
$x_2=(0,b)$, $x_3=(a-1,0)$, $x_4=(0,b-1)$. Let $\ga(x_1,x_2)$, 
 $\ga(x_2,x_3)$,  $\ga(x_3,x_4)$, $\ga(x_4,x_1)$
 be the unique subarcs in $K$, $-\bma+K$, $-\bma-\bmb+K$ and $-\bmb+K$ respectively.
 They form a   loop in $Q+K$, which encloses $P_1$, and therefore is non-contractible. 

If for the type {\bf A}, $F_\bma=(0,1)$, the square $P_1=\dfrac{x+\dot P}{n}$ has an empty intersection with $Q+K$ for both $x\in\{(-1,0 ),(0,0)\}$, the remaining argument is the same.

\begin{figure}[H]\label{forbid}
    \centering
    \includegraphics[width=.7\textwidth]{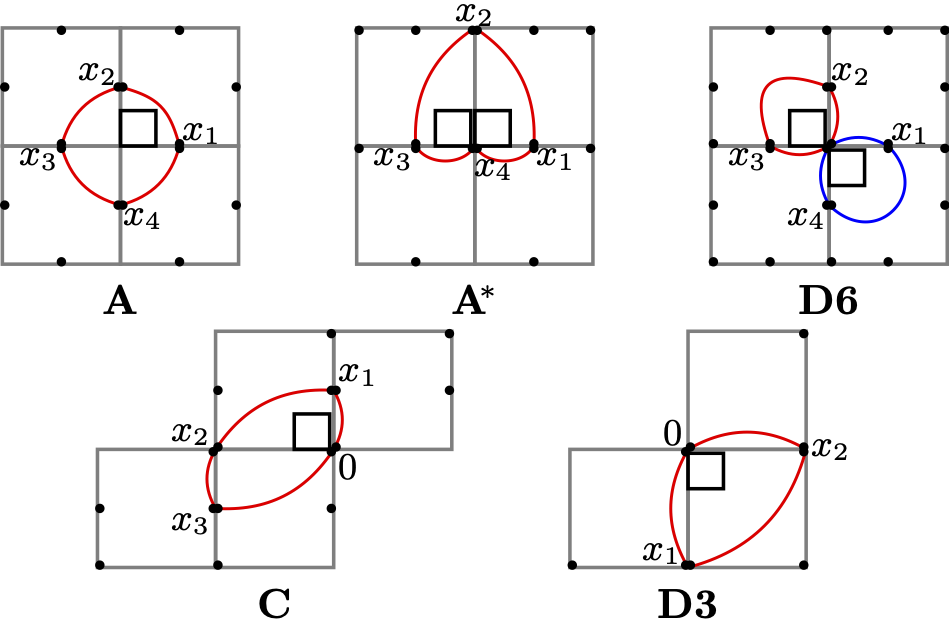}
    \caption{Construction of non-contractible loops in forbidden combinations for the types {\bf A, C} and {\bf D}.}
\end{figure}

The same way the statements c), d3) and d6) are proved.
We choose a vertex in $Q+K$, near which there is an open square $P_1=\dfrac{x+\dot P}{n}$ which has an empty intersection with $Q+K$ and then form a non-contractible loop formed by subarcs $\ga(x_i,x_j)$, which connect the boundary points of the sets $q_i+K$.
The choice of the points and the arcs is shown on the Figure \ref{forbid}.

\end{proof}

\subsection{The order and the number of corner points\\}

\begin{theorem}\label{thm:vertex_branching}
Let $K$ be a FSD.
For any corner point $a\in K$,
 $Ord(a,K)\leq 2$.
\end{theorem}

\begin{proof}\label{proof:vertex_branching}
Note that for any piece $K_\bi\ni a$, the set 
$K_\bi\mmm \{a\}=S_\bi(K\mmm \{a\})$. 
Therefore, if $Q$ is a component of $K\mmm \{a\}$, then
$Q\cap K_\bi=S_\bi(Q)$. 
Let $Q_1,\ldots, Q_k$ be the set of all these components, then $K\mmm \{a\}=\bigsqcup \limits_{j=1}^kQ_j$ and $\dd K\mmm \{a\}=\bigsqcup \limits_{j=1}^k(Q_j\cap\dd K)$, thus $\dd K_\bi\mmm \{a\}=\bigsqcup \limits_{j=1}^k(Q_j\cap\dd K_\bi)$. 
The number $k$ of the intersections in the last equality is no greater than 3, because these intersections correspond to two sides and one vertex opposite to the point $a$. 
Each of these intersections contains at most one point, and the same is true for the respective intersections  $Q_j\cap\dd K$.

Without loss of generality we may suppose  $a=(0,0)$.  Thus, $Q_1\cap K_{10}=F_{10}$ if $d_{(1,0)}\in D$, $Q_2\cap K_{01}=F_{01}$ if $d_{(0,1)}\in D$ and $Q_3\cap K_{11}=F_{11}$ if $d_{(1,1)}\in D$.

For the type {\bf A}, $F_{11}=\0$, hence $Ord(a,K)\leq 2$.
For the type {\bf B}, $F_{10}=F_{01}=\0$, therefore $Ord(a,K)=1$.
For the type {\bf C}, either $F_{10}=\0$ or $F_{01}=\0$, therefore $Ord(a,K)\leq 2$.

Suppose that the order of the corner point $a=(0,0)$ is 3.
Then the digits $(0,0)$, $(1,0)$, $(0,1)$ and $(1,1)$ are contained in $D$. 
By Lemma \ref{quadruples}, this is impossible.

This contradiction shows that in the case {\bf D}, $Ord(a,K)\leq 2$.
\end{proof}

Now we consider how many corner points may be for each of the types {\bf A, B, C, D}.

\begin{theorem}\label{thm:corner}
Let $K$ be a FSD. \\
If $K$ is of type  {\bf A}, it can have  two adjacent corner points  of orders  2,1 or 1,1,\\ or 1 corner point of   order 1 or 2, or no corner points.\\
If $K$ is of type  {\bf B}, it has four corner points of the order 1.\\
If $K$ is of type  {\bf C}, it has two corner points, where each has   order 1 or 2.\\
If $K$ is of type  {\bf D}, it may have either two corner points of orders  $\le 2$ for {\bf D6} or three corner points of  order 1 for the type {\bf D3}.
\end{theorem}
,
\begin{proof}
 $K$ contains 4 corner points if and only if it belongs to type {\bf B}. $K$  has 3 
corner points if and only if it is of type {\bf D3}. If $K$ has 2 opposite corner points then it belongs to type {\bf C} or {\bf D6}.

As it follows from  the proof of Theorem \ref{ssboundary}, if $K$ belongs to the type {\bf A}, it cannot contain  a pair of opposite corner points. 
Therefore $K$ contains at most 2 corner points which are the endpoints of some side of the square $P$. 
 
Let these  points be $(0,0)$ and $(1,0)$,
 then $\dd K=\{(0,0), (a,0), (1,0), (a,1)\}$, where $a=\dfrac{p}{n-1}$ and $p\in \{2,\ldots, n-2\}$. 
If the orders  of both points $(0,0)$ and $(1,0)$  are equal to 2, then, because each of the
pieces $\dfrac{K}{n}$ and  $\dfrac{n-1+K}{n}$ should have two neighbors,
$ \{(0,1), (n-1,1)\}\IN D$. 
In this case $F_{(1,0)}$ is  uncountable, which is impossible.
Therefore the orders of these two corner points are 2,1 or 1,1.

The case {\bf B} is obvious.
In the case {\bf C} the set $\dd K$ contains exactly 2 opposite corner points whose orders are at most 2.

Suppose $K$ belongs to the type {\bf D3} and its corner  points are $(0,0)$, $(1,0)$ and $(1,1)$. 
If $Ord((0,0),K)=2$, 
then $(0,0)$ has 2 neighbors. Therefore $(0,0)$, $(1,0)$ and $(1,1)$ are contained in $D$. By Lemma  \ref{quadruples}, this is forbidden.

The same argument may be applied to show that the point $(1,0)$ cannot have the order 2.

For the type {\bf D6}   there are exactly two opposite corner points in $\dd K$. 
The examples show that each of the combinations $1,1$; $1,2$ and $2,2$ is possible.
\end{proof} 


\begin{figure}[H]
    \centering
    \includegraphics[width=0.65\textwidth]{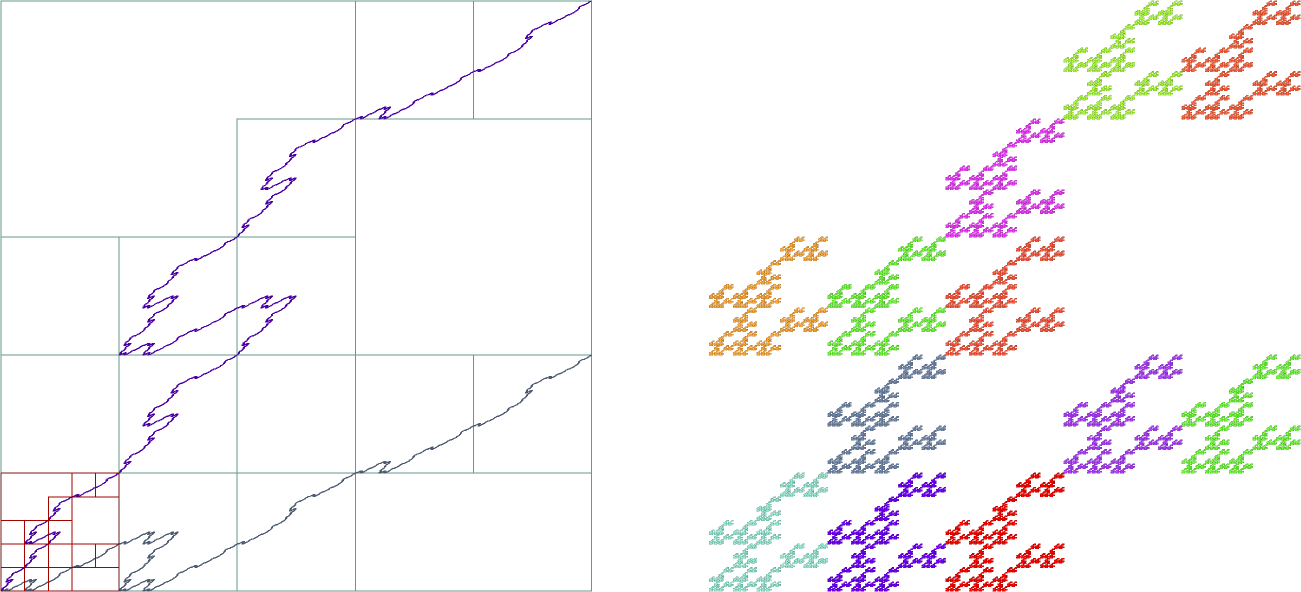}
    \caption{A corner point of order 2 for the type {\bf C} and two main arcs containing it.} 
    \label{fig:C_type_Three}
\end{figure}


\begin{theorem}\label{order}
Let $K$ be a FSD. 
For any  $x\in K$, $Ord(x,K)\le 4$.\\
If $K$ is of type {\bf D} and for any $\bi\in I^*$, $x\notin S_\bi(\dd K)$, then $Ord(x,K)\le 3$.\\
If $K$ is of type {\bf D3} and $x\in \dd K_\bi$ for some $\bi\in I^*$ then  $Ord(x, K)\le 3$.
\end{theorem}

\begin{proof}\label{proof:point_branching}
Let $Ord(x,K)=l$.
Then there are $l$ Jordan arcs $\{\ga_i\}_{i=1}^l$  in $K$, with a common endpoint $x$ and such that the arcs $\ga_i\mmm \{x\}$   are disjoint.
There is an integer $p>0$, such that for each $\ga_i$, the diameter $|\ga_i|$ is greater than $\sqrt{2}n^{-p}$.
Bearing in mind Lemma \ref{quadruples}, one sees  that there are 1, 2 or 3 multi-indices  $\bi=i_1i_2..i_p$ such that  $x\in S_\bi(K)$.

In the case (1), when such $\bi$ is unique, each arc $\{\ga_i\}$ passes to one of the neighbour squares of  order $n^p$ for $K_\bi$ through one of the points  of $\dd K_\bi$ (see Figure \ref{fig:case1}). 
Thus the number $l$ is no greater than the number of neighbour squares for $K_\bi$  and  the  number of the points of $\dd K_\bi$. 

A similar argument applies to the case (2) when $x\in K_\bi\cap K_\bj$ and lies on the interior of a common side of squares $S_\bi(P)$ and $S_\bj(P)$ (see Figure \ref{fig:case2}) and to the case (3) when $x$ is a common corner point of two or three squares (see Figure \ref{fig:case3}). 
In these cases the number $l$ is no greater than the number of neighbour squares for those $K_\bi$, that contain $x$ and  respective points of their boundaries.\textbf{(1)}
Let $x\in \dot K_\bi$.
For the types {\bf A, B, C} the number $l\leq 4$.

For the type {\bf D} the number $l\le 3$. 
Indeed, choosing any quadruple among 6 possible neighbors of a piece $K_\bi$, we get a forbidden combination. 
\begin{figure}[H]
    \centering
    \includegraphics[width=0.75\textwidth]{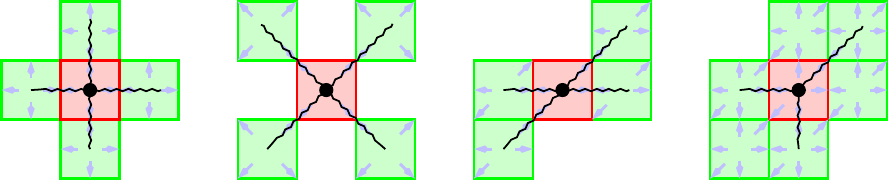}
    \caption{The orders of interior points for the types \textbf{A}, \textbf{B}, \textbf{C} and \textbf{D}}
    \label{fig:case1}
\end{figure}

\begin{figure}[H]
    \centering
    \includegraphics[width=0.6\textwidth]{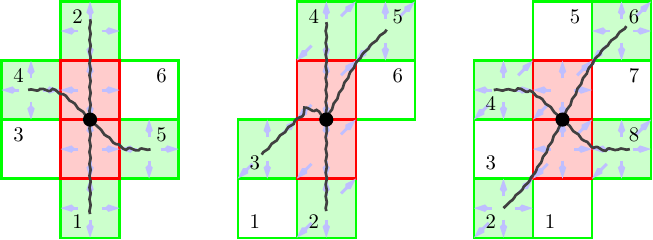}
    \caption{Accessible neighbors of a point $x$ on a common side for the types \textbf{A}, \textbf{C} and \textbf{D6}}. 
    \label{fig:case2}
\end{figure}

\textbf{(2)}
Let $x\in S_\bi(\dd K)$ be an inner point of a common side  of $S_\bi(P)$ and $S_\bj(P)$.
This is possible in the cases \textbf{A}, \textbf{C} and \textbf{D6} only.

In these cases the number of accessible neighbors of $S_\bi(P)\cup S_\bj(P)$ cannot exceed $4$ by Lemma \ref{quadruples}.

\begin{figure}[H]
    \centering
    \includegraphics[width=0.8\textwidth]{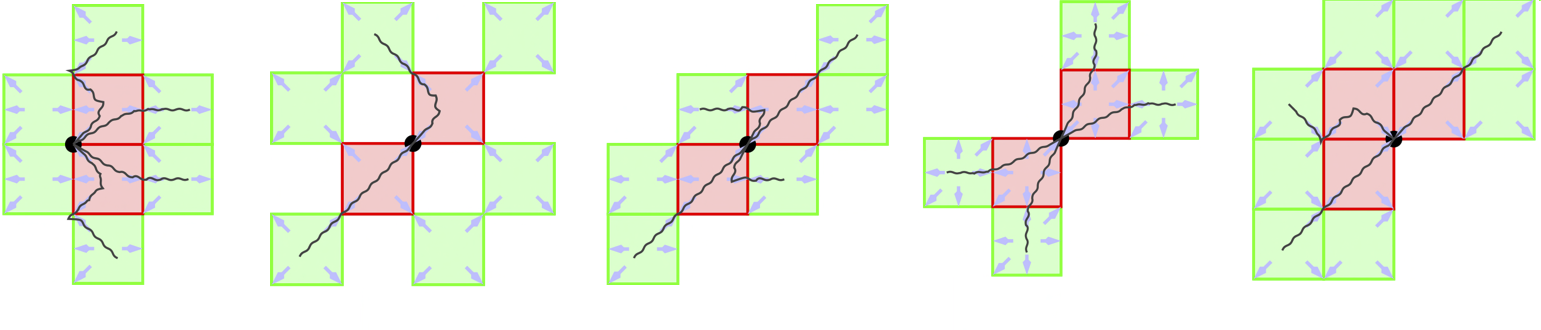}
    \caption{The orders of corner points for   {\bf A, B, C, D6, D3}.}
    \label{fig:case3}
\end{figure}

\textbf{(3)}
Let $x$ be a vertex of $S_\bi(P)$.
In the cases \textbf{A},  \textbf{B}, \textbf{C} and \textbf{D6} the point $x$ belongs to at most two pieces $K_\bi, K_\bj$, and in the case \textbf{D3} the point $x$ can belong to at most 3 pieces $K_\bi, K_\bj, K_\bk$ as it is shown on the Figure \ref{fig:case3}.

By Theorem \ref{thm:vertex_branching}, the corner point $x$ has $Ord\leq2$ in each of the pieces.
Moreover, for type \textbf{B} and \textbf{D3} the order of corner points is  1.

Consequently, for the corner point $Ord(x,K)\leq4$ in cases {\bf A, C, D6}, $Ord(x,K)\leq3$ in case {\bf D3} and 
$Ord(x,K)\leq2$ in case {\bf B}.
\end{proof}

\section{Classification theorem for fractal square dendrites}

\begin{proposition}\label{lem:d4bound}
Suppose $K$ is of the type  {\bf D6}  and $\gamma$ is its main tree.\\ 
 (i)  for any $x\in\dd K$, $Ord(x,\gamma)\leq2$;\\
(ii) for any $x\in\ga$, $Ord(x,K)\leq3$.
\end{proposition}

\begin{proof}
(i). For the corner points of $K$ the inequality $Ord(x,\gamma)\leq2$ follows from Proposition \ref{thm:corner}.

For non-corner points $x\in\dd K$ the argument of the proof of Theorem \ref{order} implies that the order of $x$ with respect to $\ga$ is less or equal to the number of neighbors of the piece $K_i$. One can see from the Figure
\ref{fig:case2} that any piece $K_i$, marked on this Figure by red color, can have at most 2 neighbors, otherwise a forbidden set of digits would  appear. 
Thus $Ord(x,\gamma)\leq Ord(x,K)\leq2$.

\begin{figure}[H]
    \centering
    \includegraphics[width=0.75\textwidth]{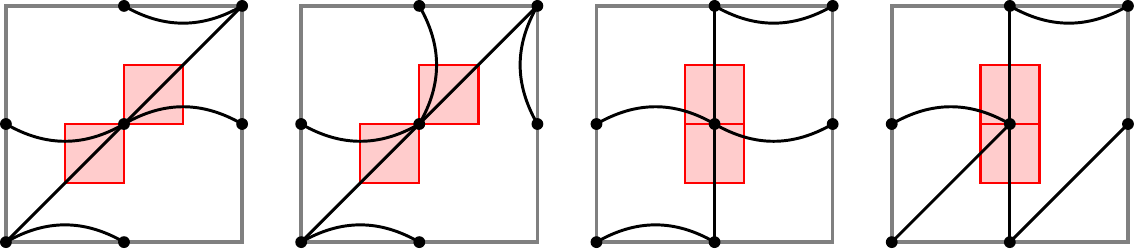}
    \caption{ 
  Imaginable arrangements for the main tree $\ga$ under the assumption that $Ord(x,\gamma)=4$ for the type {\bf D6}. The point $x$ is in the center.}
    \label{fig:d6ord3full}
\end{figure}

(ii)  Suppose there is $x\in\ga$ such that $Ord(x,\gamma)=4.$
Theorem \ref{order} implies that  $x=S_i(K)\cap S_j(K)$ for some $i,j\in I$.\\
Consider first the case when $x$ is a common corner point of two pieces $S_i(K)$ and $S_j(K)$,
so $x=S_i(b)=S_j(a)$, where $a=(0,0)$ and $b=(1,1)$. Denote also $F_{0,-1}=\{c\}$,
$F_{0,1}=\{d\}$,
$F_{-1,0}=\{e\}$, $F_{1,0}=\{f\}$.

\begin{figure}[H]
    \centering \includegraphics[width=0.5\textwidth]{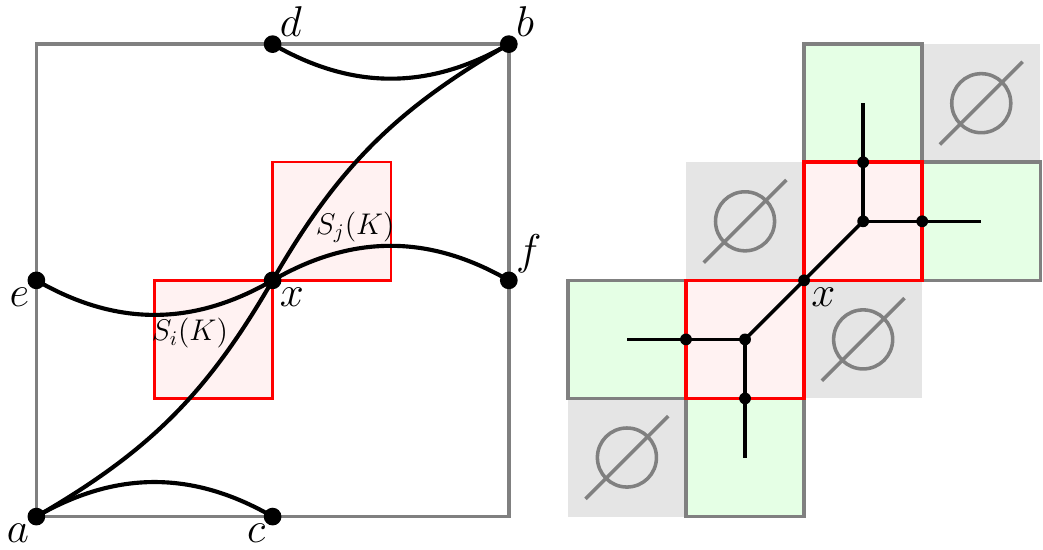}
\end{figure}

For the points $a,b$ it follows from Theorems \ref{thm:vertex_branching}, \ref{order} that $Ord(a,\ga)=Ord(b,\ga)=2$. 
Therefore the set $\ga\mmm\{a,b\}$  consists of three components. 
We denote the closures of these components by $\ga_0$, $\ga_a$, $\ga_b$, so that $\ga_0\cap\ga_a=\{a\}$ and $\ga_0\cap \ga_b=\{b\}$. The component $\ga_0$ contains the arc $\ga(a,b)$ which divides $P$ into two parts. One of these contains the points $c$ and $f$, the other contains
$e$ and $d$, therefore $x\in \ga(a,b)$ and the set $\ga_0$ contains one of the pairs of points $c,d$ or $c,f$ or $e,f$ or $e,d$.

Consider the case when 
$\ga_0\NI\{e,f\}$. Then
$\ga_a=\ga(a,c)$ and $\ga_b=\ga(b,d)$.

Then the set $\ga_0$ is the union of subarcs $\ga(a,x)$, $\ga(f,x)$, $\ga(b,x)$ and $\ga(e,x)$.

These four subarcs intersect the set $S_i(\dd K)\cup S_j(\dd K)$ at four different points $S_i(c)$, $S_i(e)$, $S_j(f)$, $S_j(d)$. 
Therefore the  intersections of these four subarcs with $S_i( K)\cup S_j(K)$ should be different subarcs $S_i(\ga(c,b))$, 
 $S_i(\ga(e,b))$, $S_j(\ga(f,a))$, $S_j(\ga(d))$ whose interiors lie in different components of $\ga_0\mmm\{x\}$. 
This is impossible, because $\ga(a,f)\cap\ga(a,d)=\ga(a,x)$.\\

Consider now the case when $x$ lies on a common side of two pieces $S_i(K)$ and $S_j(K)$. In this case
$F_{0,-1}=\{c\}$,
$F_{0,1}=\{d\}$
so $x=S_i(d)=S_j(c)$.  Denote also $F_{-1,-1}=\{a\}$,
$F_{1,1}=\{b\}$,
$F_{-1,0}=\{e\}$, $F_{1,0}=\{f\}$.

For the points $c,d$ it follows from Theorems \ref{thm:vertex_branching}, \ref{order} that $Ord(c,\ga)=Ord(d,\ga)=2$. 

Therefore the set $\ga\mmm\{c,d\}$  consists of three components. 
We denote the closures of these components by $\ga_0$, $\ga_c$, $\ga_d$, so that $\ga_0\cap\ga_c=\{c\}$ and $\ga_0\cap \ga_d=\{d\}$. The component $\ga_0$ contains the arc $\ga(c,d)$ which divides $P$ into two parts. One of these parts contains the points $b$ and $f$, the other contains
$e$ and $a$, therefore $x\in \ga(c,d)$ and the set $\ga_0$ contains one of the pairs of points $a,b$ or $a,e$ or $e,f$ or $b,f$. 

\begin{figure}[H]
    \centering
    \includegraphics[width=0.5\textwidth]{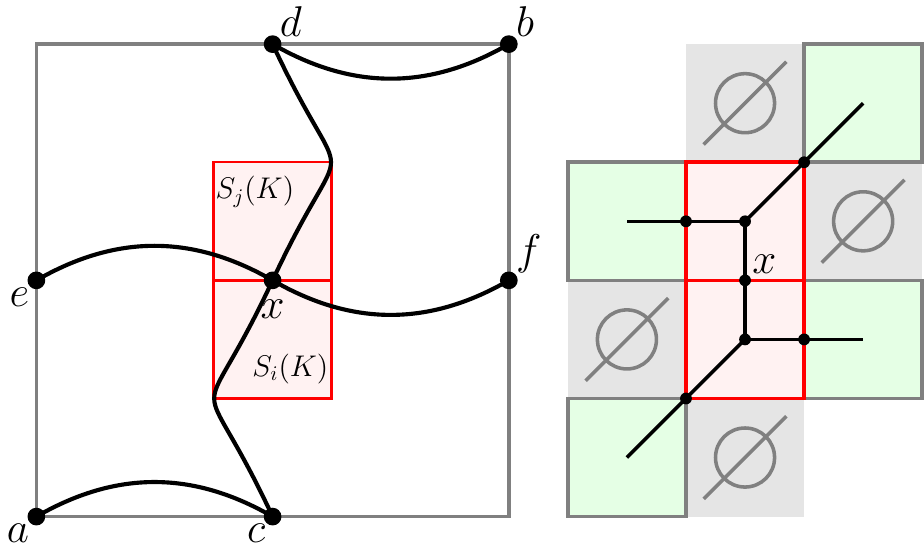}
\end{figure}

Consider the case when 
$\ga_0\NI\{e,f\}$. Then
$\ga_c=\ga(a,c)$ and $\ga_d=\ga(b,d)$.

Then the set $\ga_0$ is the union of subarcs $\ga(c,x)$, $\ga(f,x)$, $\ga(d,x)$ and $\ga(e,x)$.

These four subarcs intersect the set $S_i(\dd K)\cup S_j(\dd K)$ at four different points $S_i(e)$, $S_i(b)$, $S_j(f)$, $S_j(a)$. 
Therefore the  intersections of these four subarcs with $S_i( K)\cup S_j(K)$ should be different subarcs $S_j(\ga(c,b))$, 
 $S_j(\ga(c,e))$, $S_i(\ga(a,d))$, $S_j(\ga(d,f))$ whose interiors lie in different components of $\ga_0\mmm\{x\}$. 
This is impossible, because $\ga(a,d)\cap\ga(d,f)=\ga(d,x)$.
\end{proof}

\begin{theorem}\label{thm:7trees}
There are only 7 types of main trees for FSD's, as it is shown in the table below:
\end{theorem}

\begin{figure}[H]
    \centering \Large {\bf
    1. \includegraphics[width=0.13\textwidth]{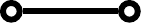}
    \hfill
    2. \includegraphics[width=0.13\textwidth]{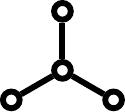}
    \hfill
    3. \includegraphics[width=0.12\textwidth]{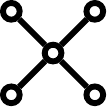}
    \hfill
    4. \includegraphics[width=0.13\textwidth]{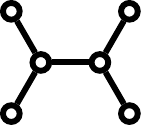}\\
    \smallskip
    5. \includegraphics[width=0.2\textwidth]{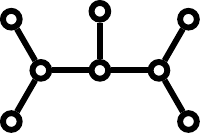}
    \hfill
    6. \includegraphics[width=0.25\textwidth]{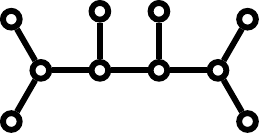}
    \hfill
    7. \includegraphics[width=0.18\textwidth]{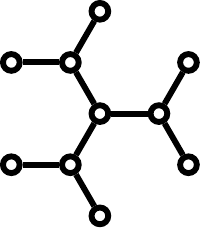}}
    \caption{Seven types of  main trees.}
    \label{fig:7trees}
\end{figure}


\begin{proof}
Let $K$ be a FSD with the main tree $\ga$. All its endpoints lie in 
$\dd K$.

Suppose $K$ belongs to the types {\bf A} or {\bf C}. By Theorem \ref{ssboundary}, the main tree $\gamma$ in this case can contain 2, 3 or 4 end points, which gives  the patterns  
1, 2, 3 and 4 on the Figure \ref{fig:7trees}.

If $K$ belongs to the type {\bf B},
by Theorems \ref{ssboundary} and  \ref{thm:corner} the main tree has exactly four endpoints hence it has the patterns   3 and 4.

If $K$ is of type {\bf D3}, then by Theorems \ref{ssboundary} and \ref{thm:corner} the main tree $\gamma$ has exactly $3$ endpoints, hence it  has the pattern 2.

If $K$ belongs to the type {\bf D6}, then by Theorem \ref{ssboundary} and  Lemma \ref{lem:d4bound} the main tree $\gamma$ may havedpoints fron 2 to 6 and for any  $x\in \gamma $, $ Ord(x,\gamma)\leq3$, therefore $\gamma$ cannot have the form different from patterns  1, 2, 4, 5, 6 or 7 from the Figure \ref{fig:7trees}.
\end{proof}

According to our search, the fractal square dendrites may be divided to 16 types depending on the form of their main tree and the placement of the points of $\dd K$ on subarcs of the main tree as it is  shown on the table below.

\begin{figure}[H]
    \centering
    \includegraphics[scale=1.2]{mt1.pdf}
    \vspace{0.3cm}
    \vfill
    \fbox{
    \includegraphics[width=0.22\textwidth]{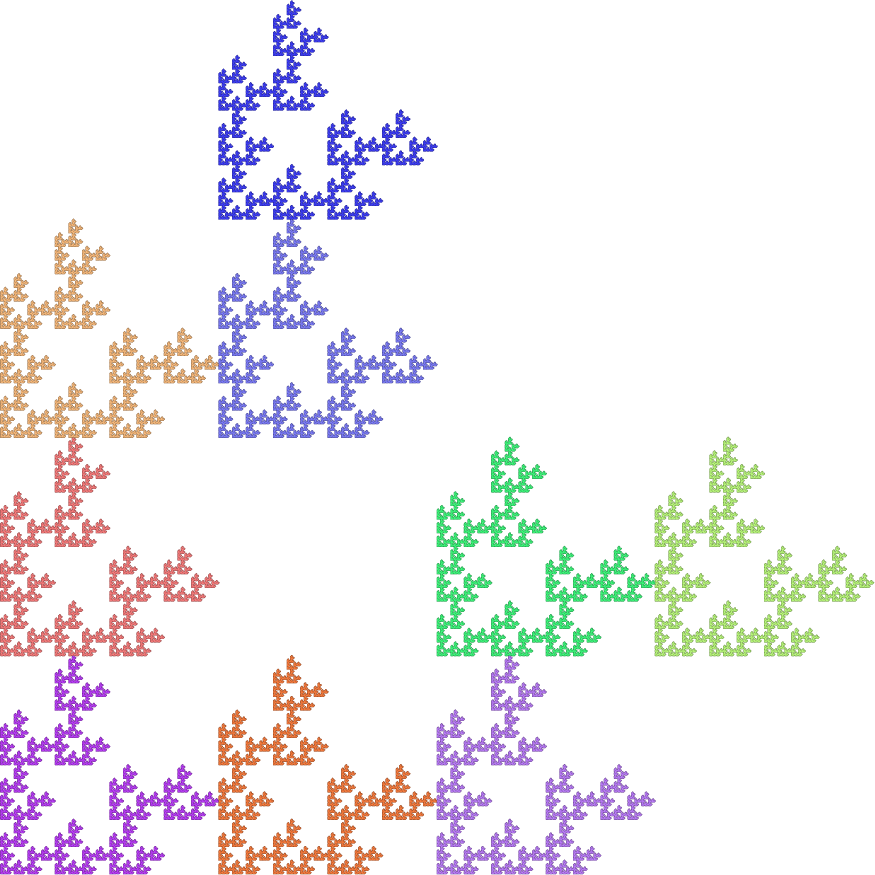}
    \includegraphics[width=0.22\textwidth]{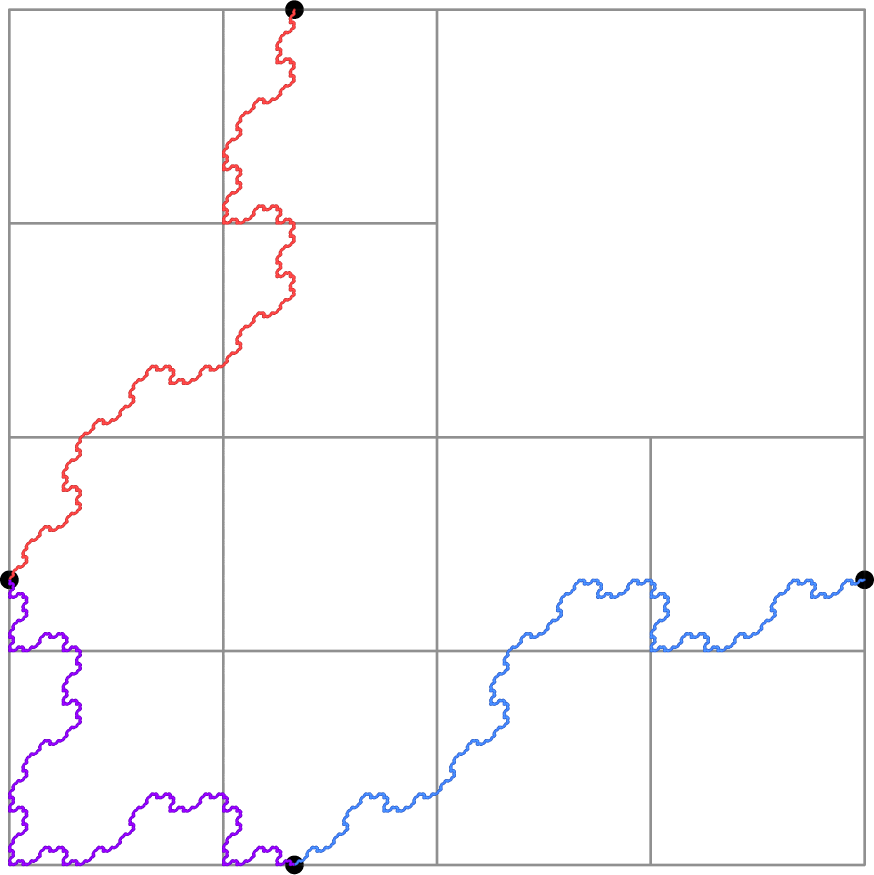}}
    \hfill
    \fbox{
    \includegraphics[width=0.22\textwidth]{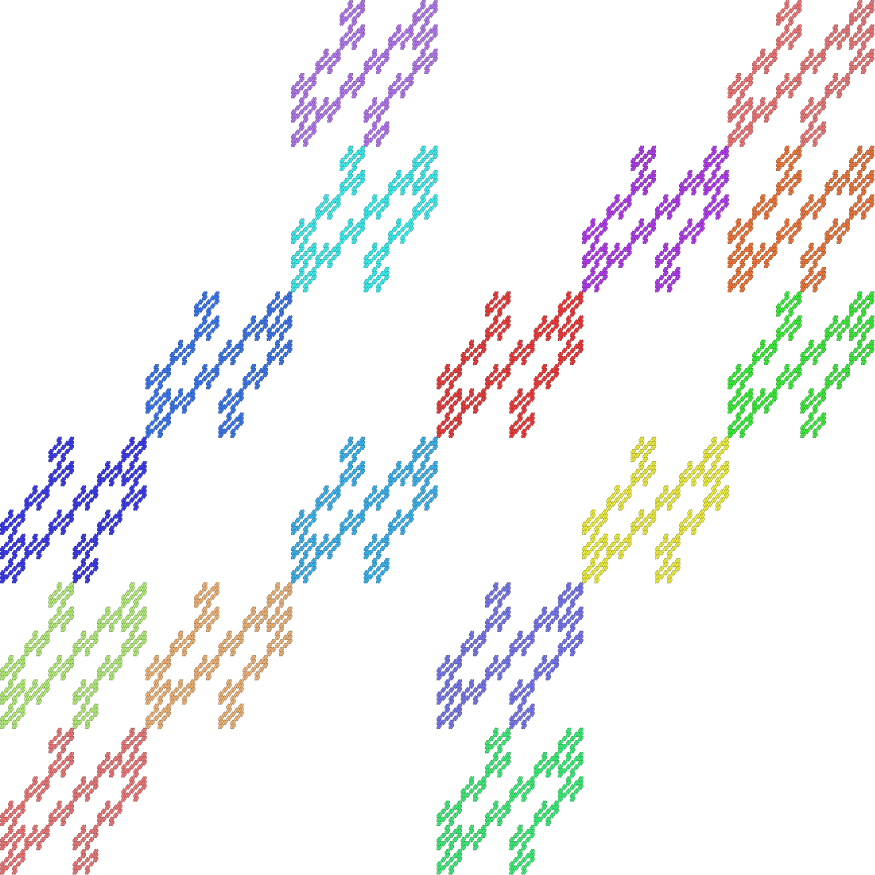}
    \includegraphics[width=0.22\textwidth]{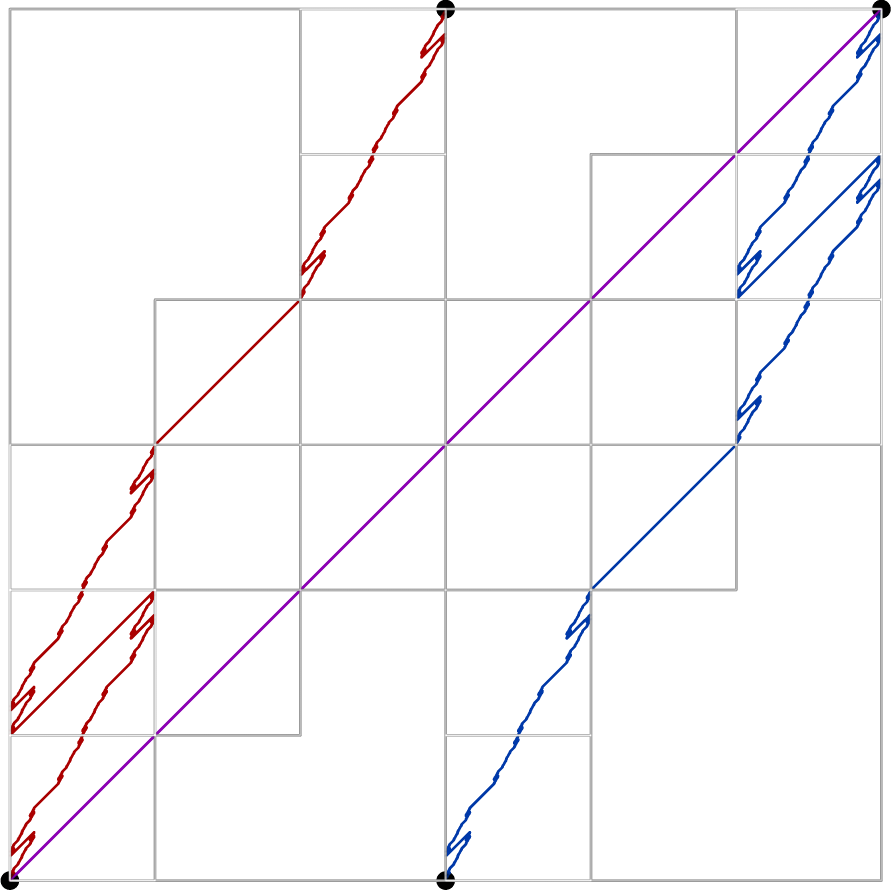}}
    \caption{Classes 1 and 2. The tree types are \textbf{1-A} and \textbf{1-C}. }
    \label{fig:tree1}
\end{figure}

\begin{figure}[H]
    \centering
    \includegraphics[scale=1.2]{mt2.pdf}
    \vspace{0.3cm}\vfill
    \fbox{
    \includegraphics[width=0.22\textwidth]{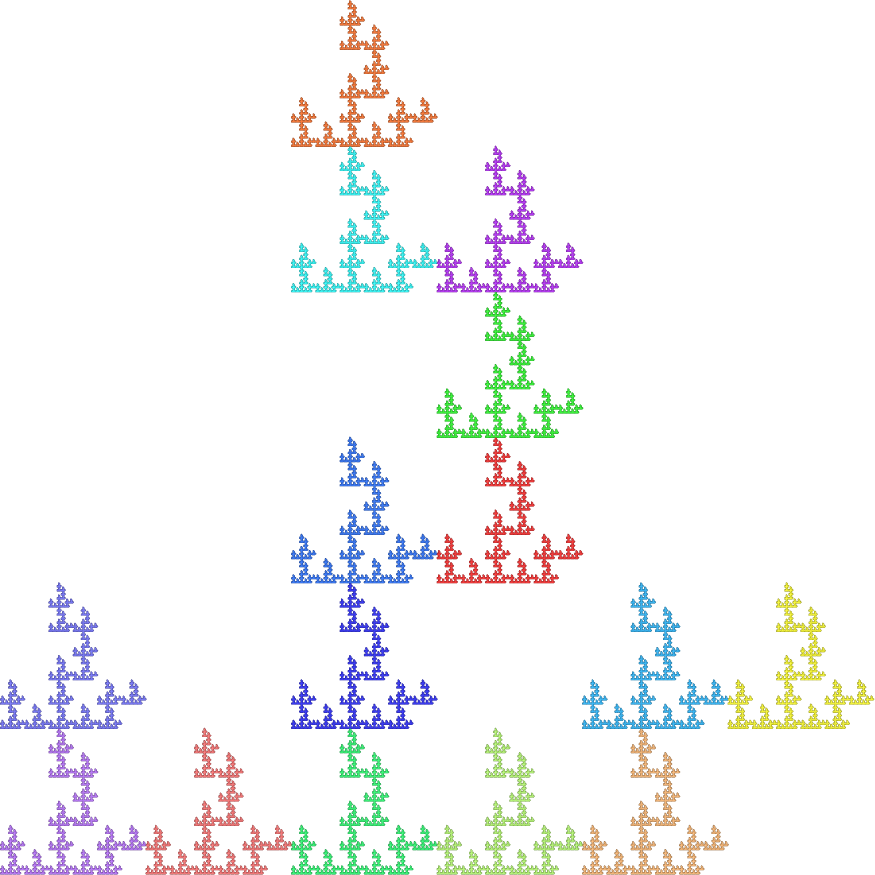}
    \includegraphics[width=0.22\textwidth]{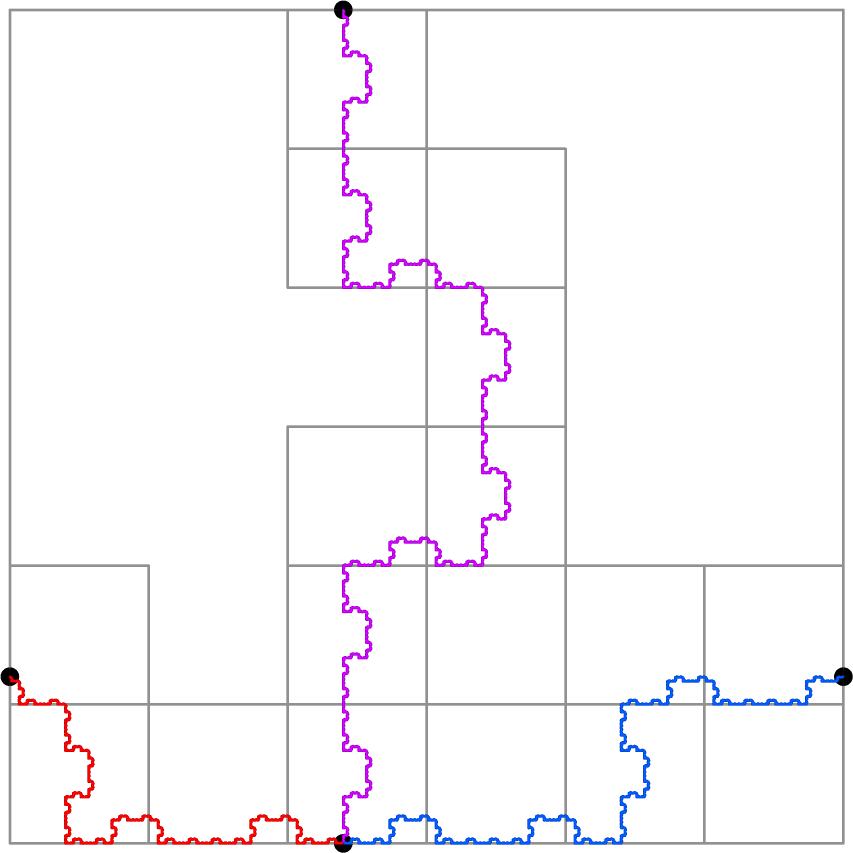}}
    \hfill
    \fbox{
    \includegraphics[width=0.22\textwidth]{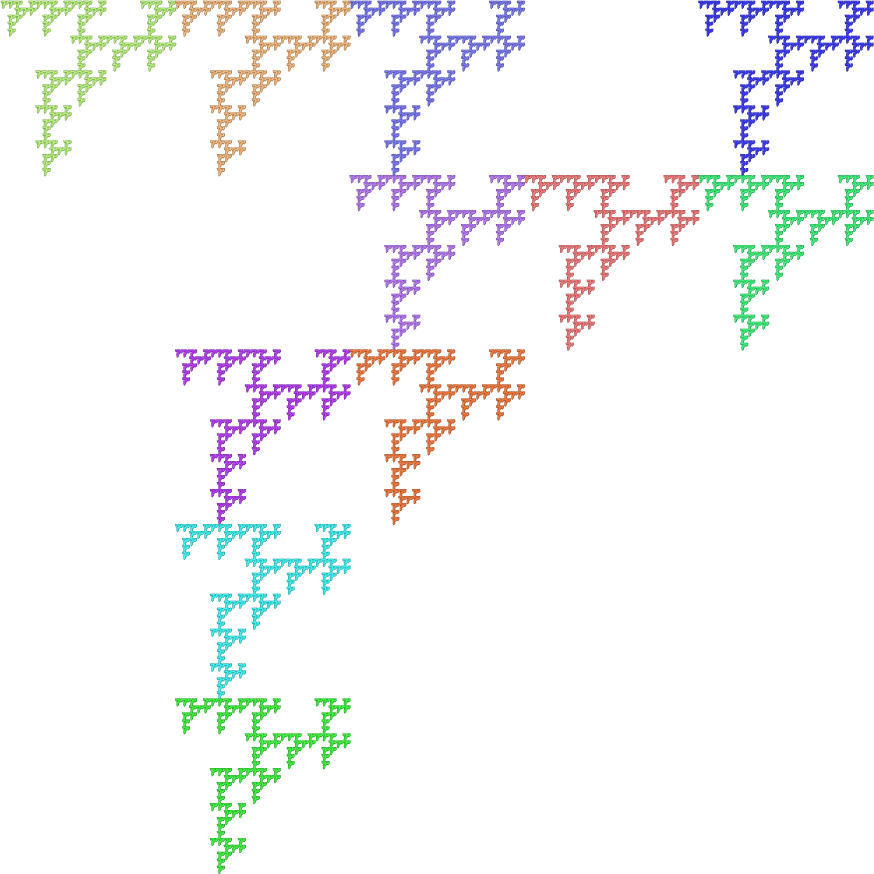}
    \includegraphics[width=0.22\textwidth]{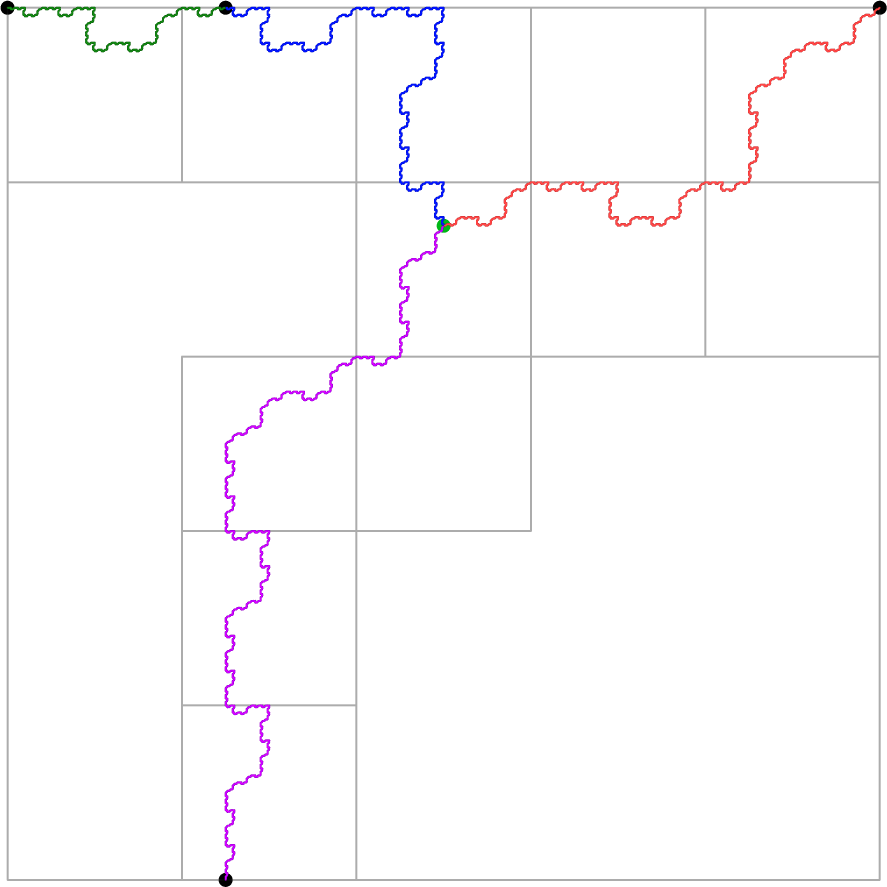}}
    \vspace{0.3cm}\vfill
    \fbox{
    \includegraphics[width=0.22\textwidth]{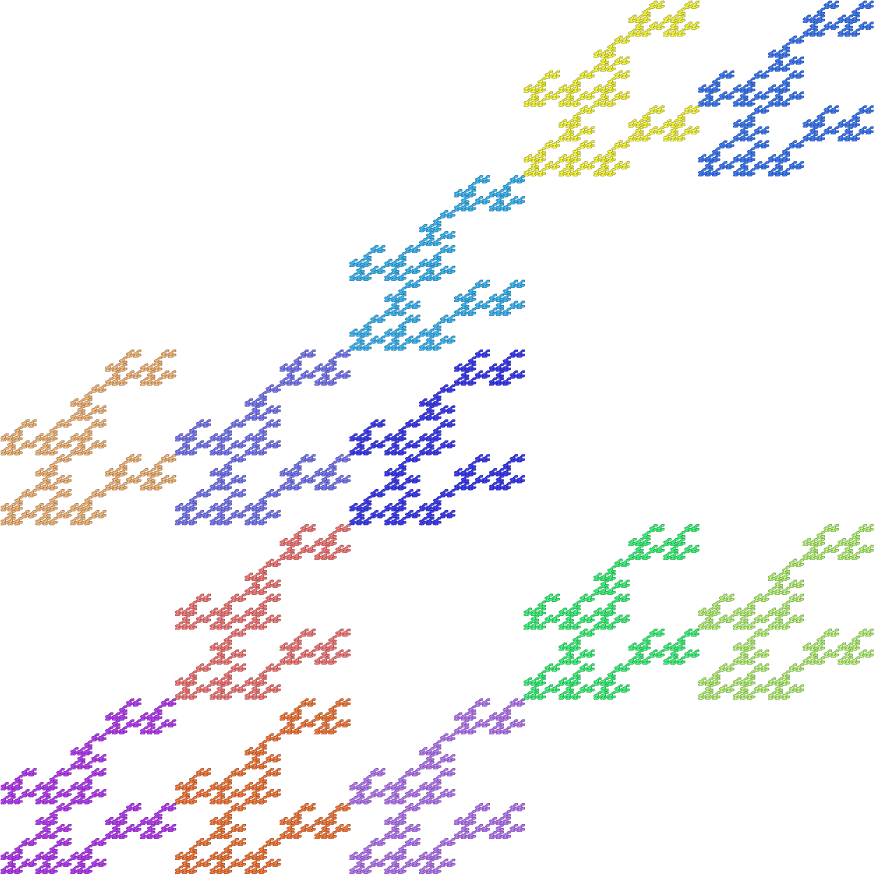}
    \includegraphics[width=0.22\textwidth]{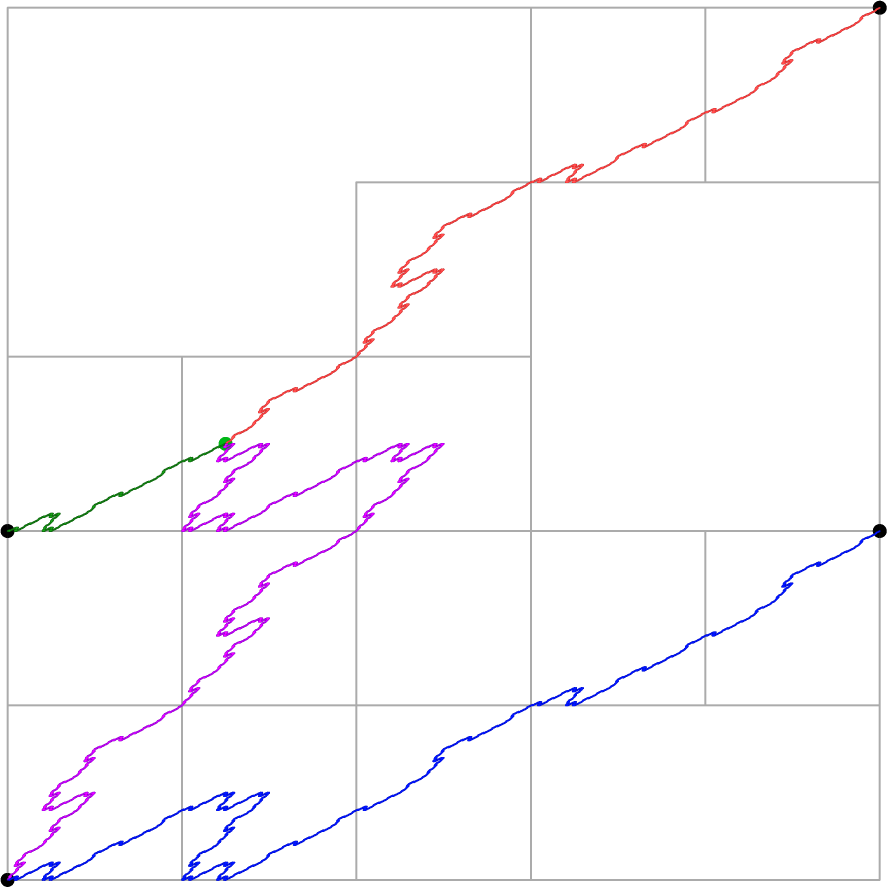}}
    \hfill
    \fbox{
    \includegraphics[width=0.22\textwidth]{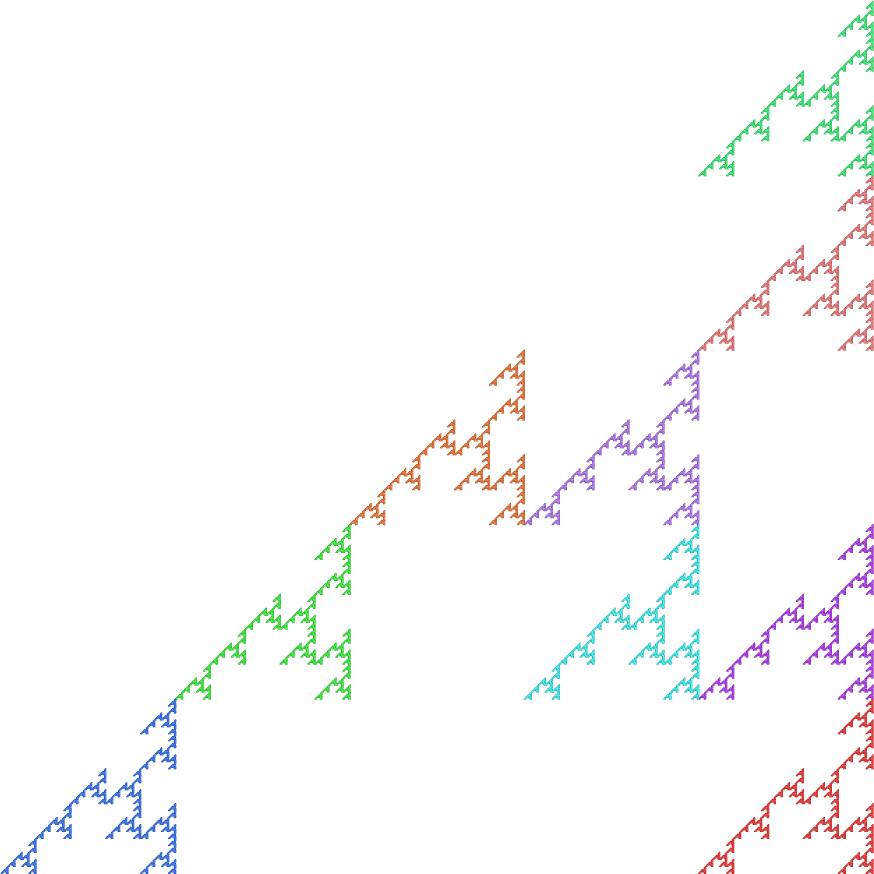}
    \includegraphics[width=0.22\textwidth]{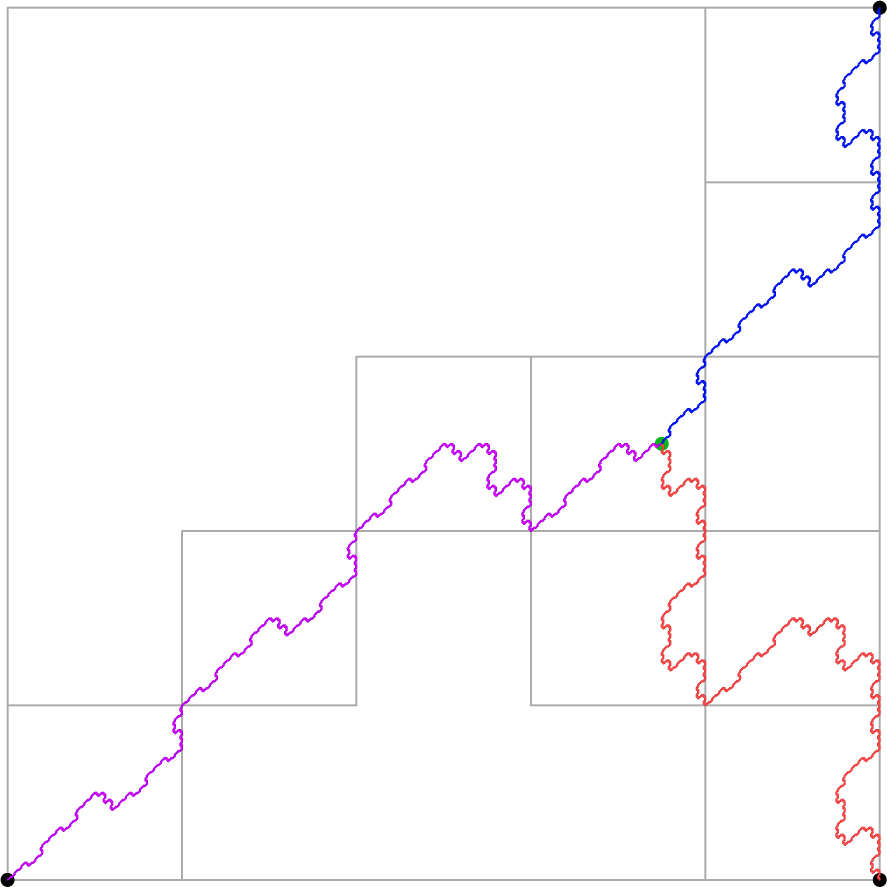}}
    \caption{Classes 3 ({\bf 2-A}), 4 ({\bf 2-A}), 5 ({\bf 2-D6}) and 6 ({\bf 2-D3})}
    \label{fig:tree2}
\end{figure}

\begin{figure}[H]
    \centering
    \hspace{0.12\textwidth}
    \includegraphics[scale=1.2]{mt3.pdf}
    \hfill
    \fbox{
    \includegraphics[width=0.22\textwidth]{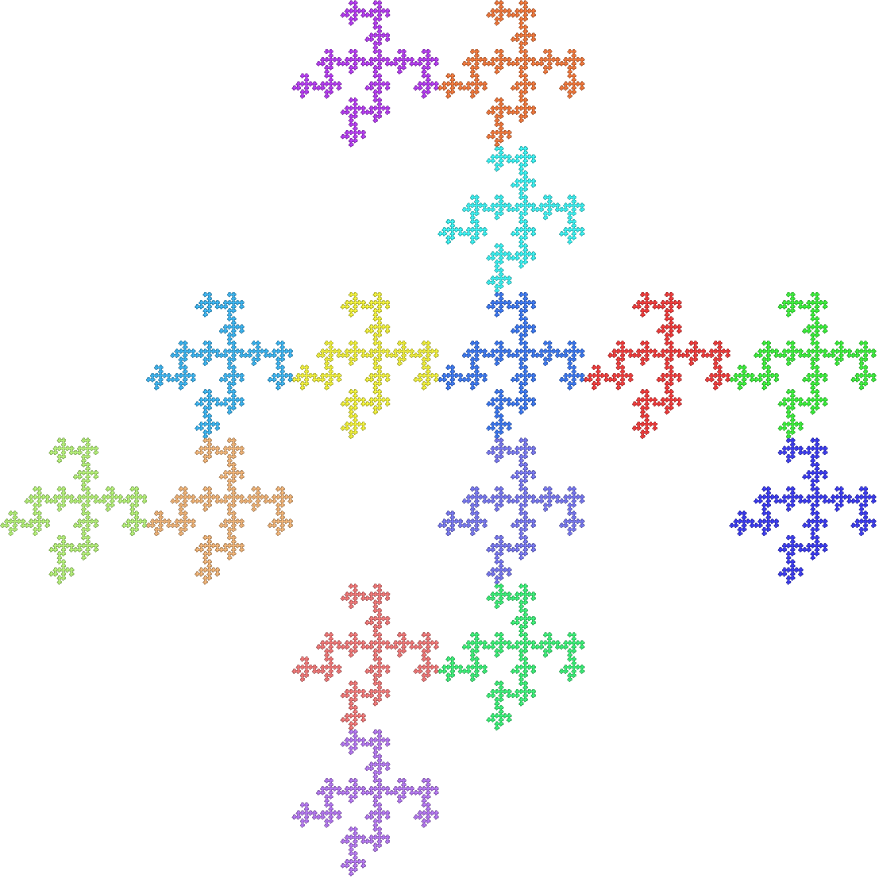}
    \includegraphics[width=0.22\textwidth]{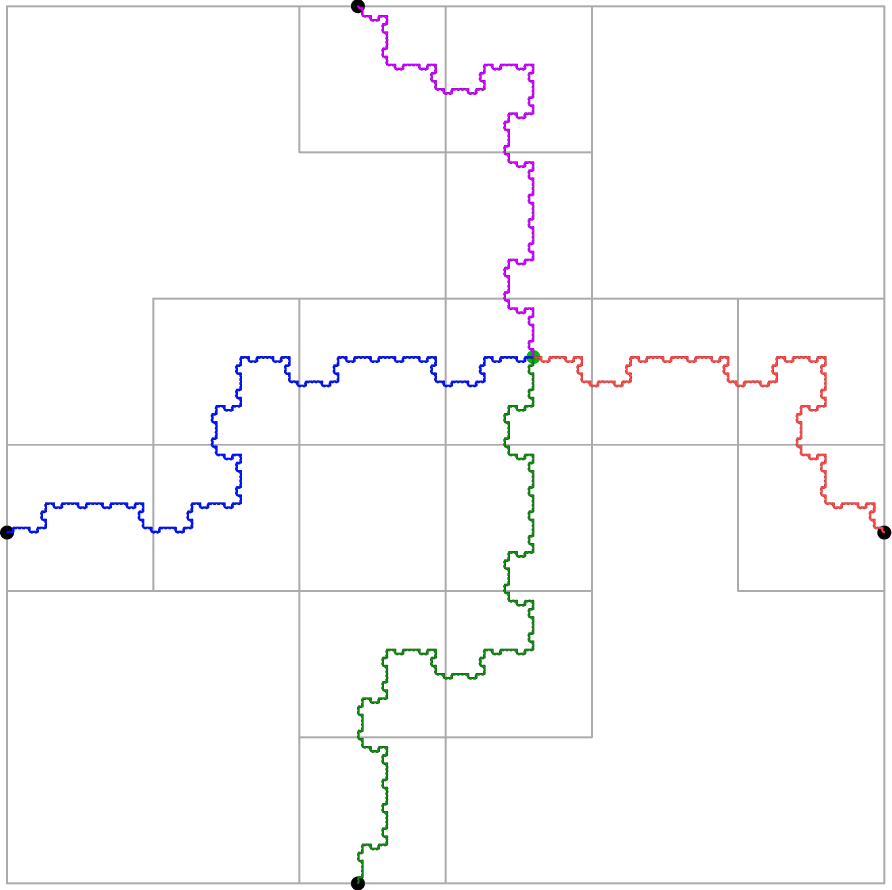}}
    \vspace{0.3cm}\vfill
    \fbox{
    \includegraphics[width=0.22\textwidth]{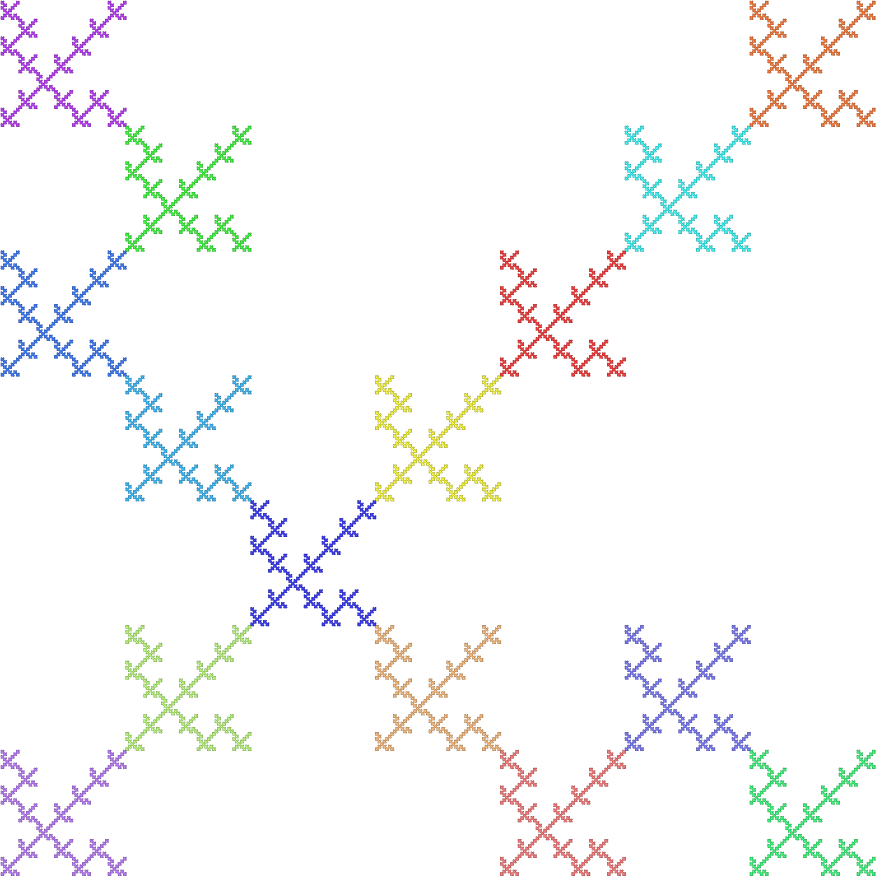}
    \includegraphics[width=0.22\textwidth]{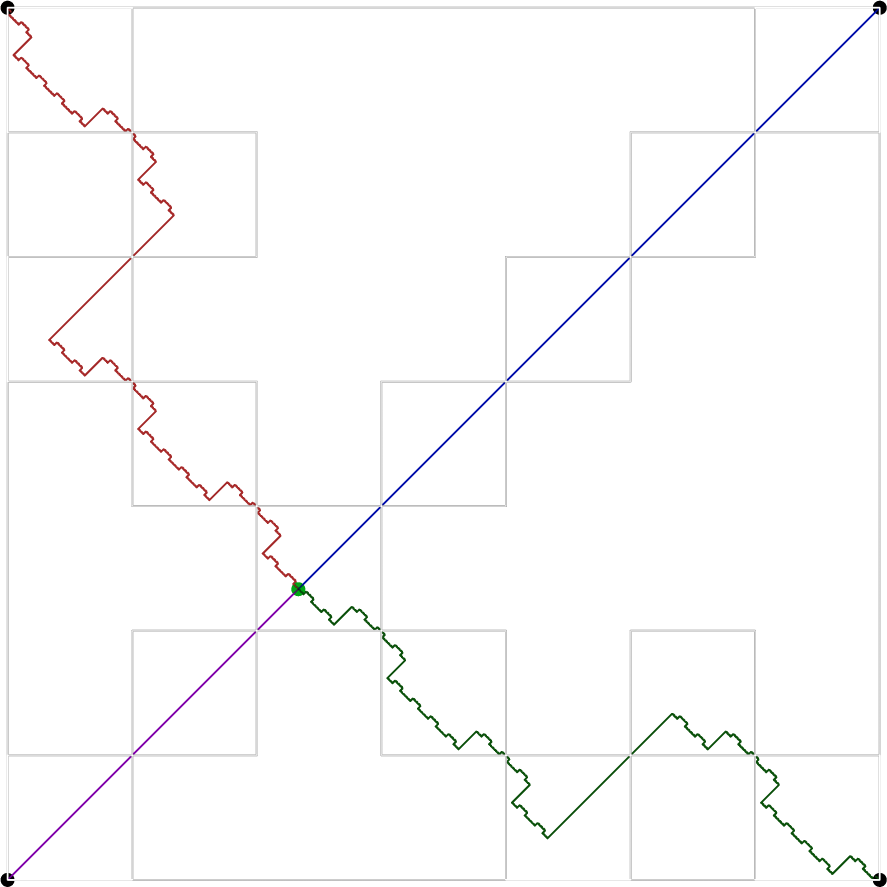}}
    \hfill
    \fbox{
    \includegraphics[width=0.22\textwidth]{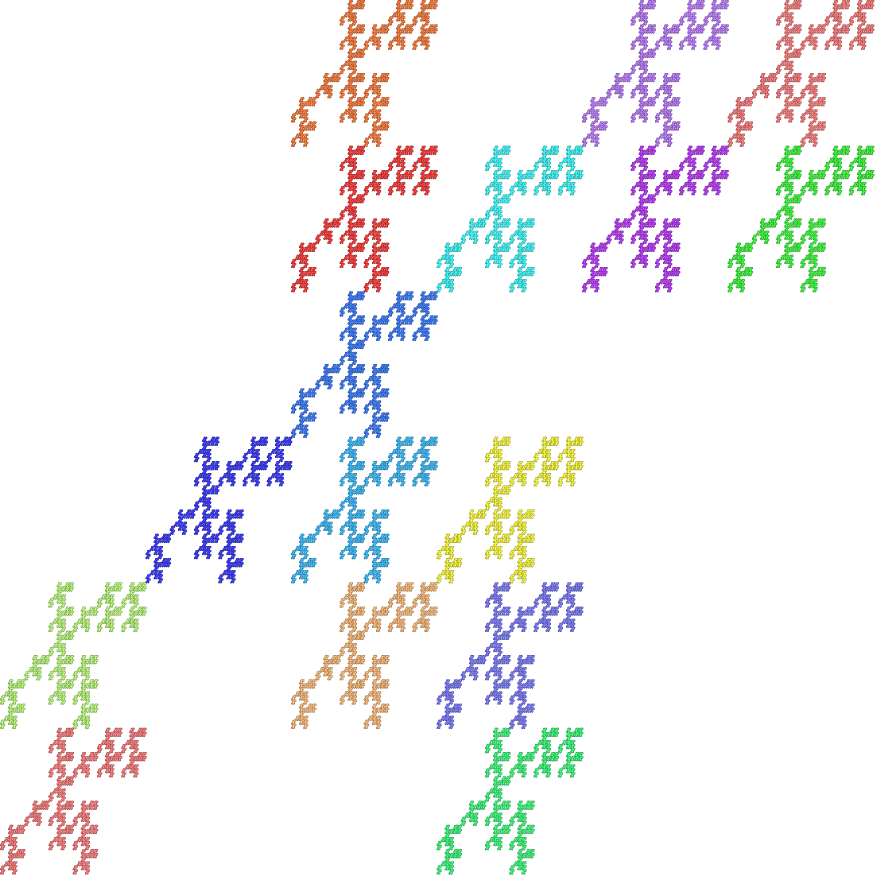}
    \includegraphics[width=0.22\textwidth]{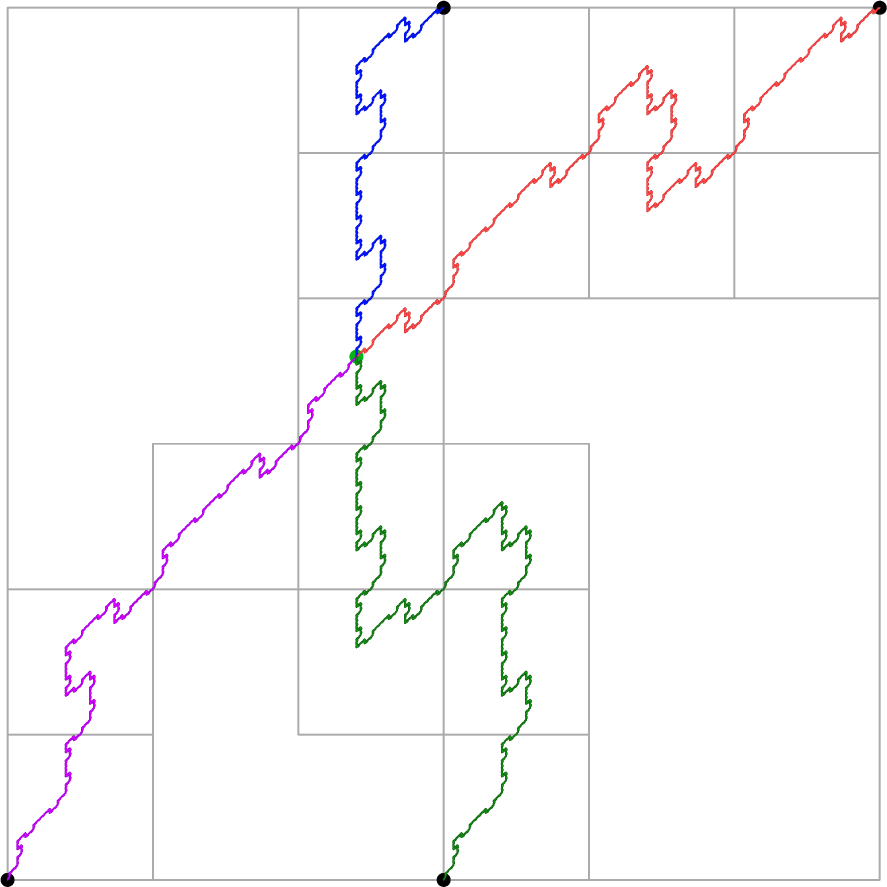}}
    \caption{Classes 7({\bf 3-A}), 8 ({\bf 3-B}) and 9({\bf 3-C})}
    \label{fig:tree3}
\end{figure}
 
\begin{figure}[H]
    \centering
    \includegraphics[scale=1.2]{mt4.pdf}
    \vspace{0.5cm}\vfill
    \fbox{
    \includegraphics[width=0.22\textwidth]{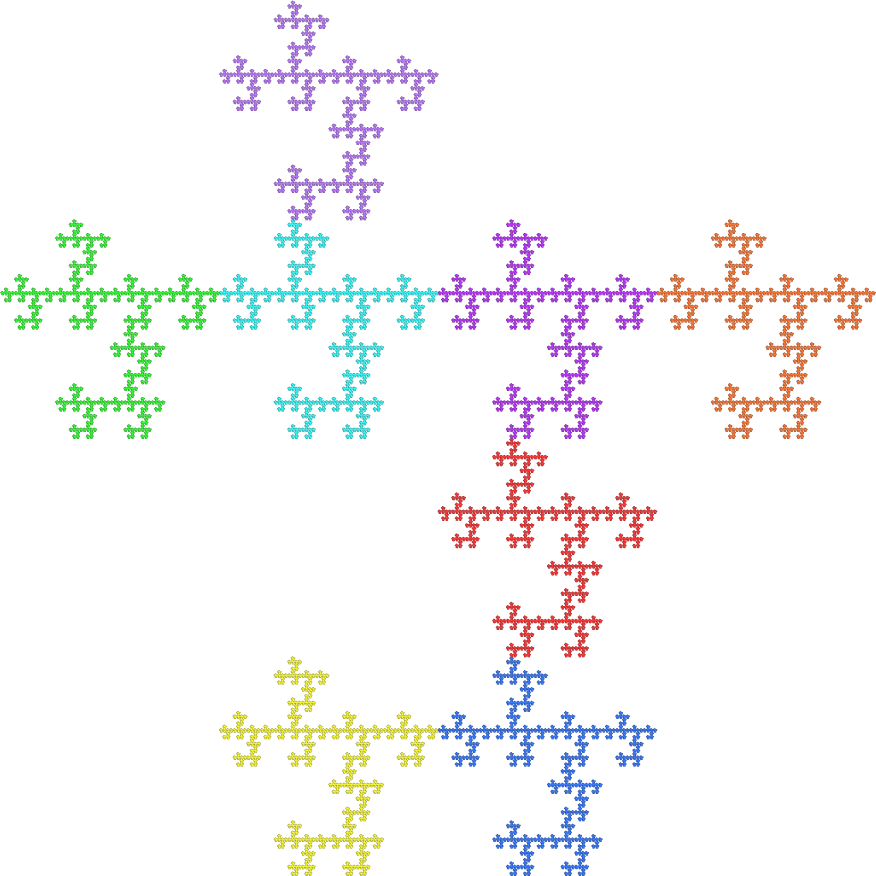}
    \includegraphics[width=0.22\textwidth]{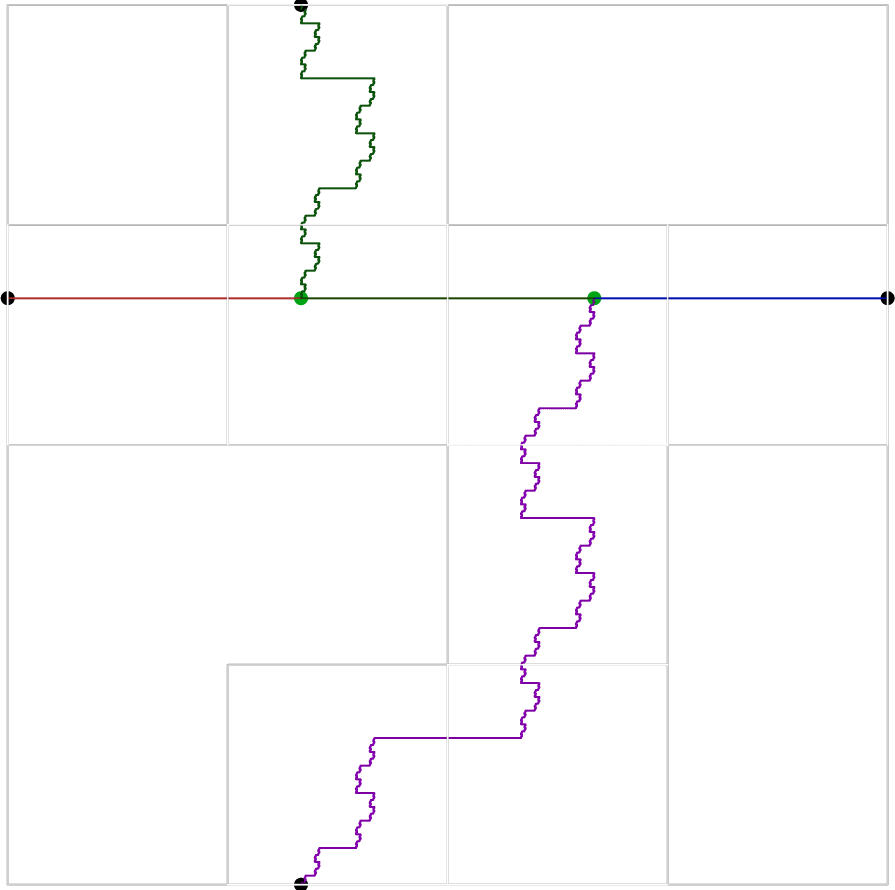}}
    \hfill
    \fbox{
    \includegraphics[width=0.22\textwidth]{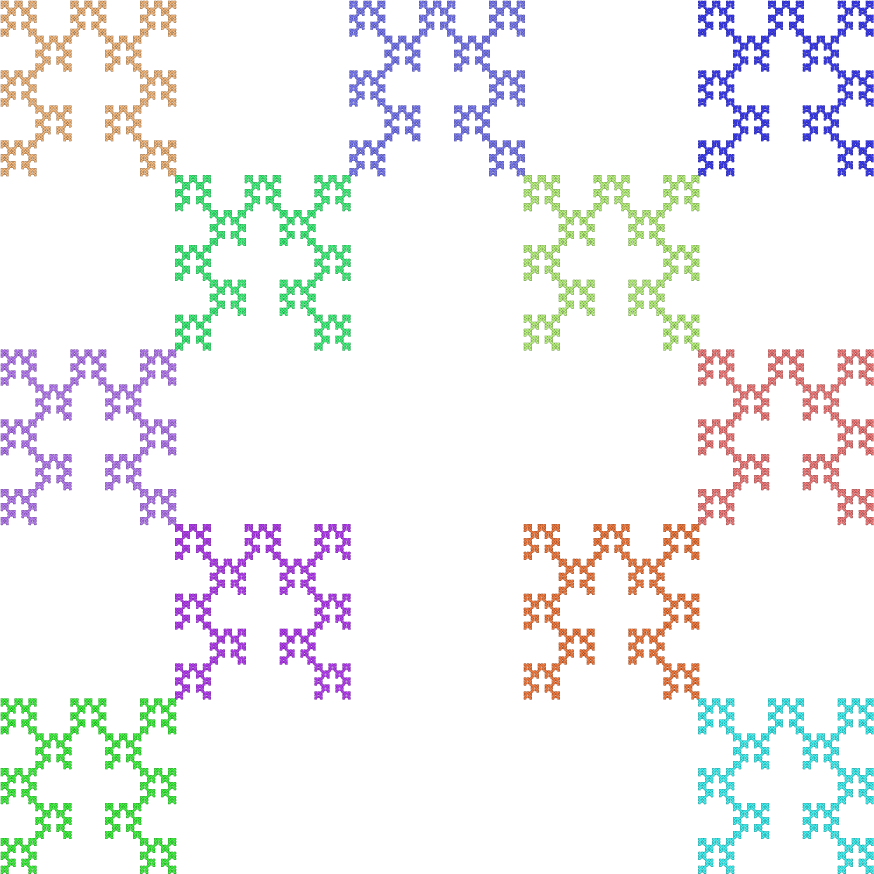}
    \includegraphics[width=0.22\textwidth]{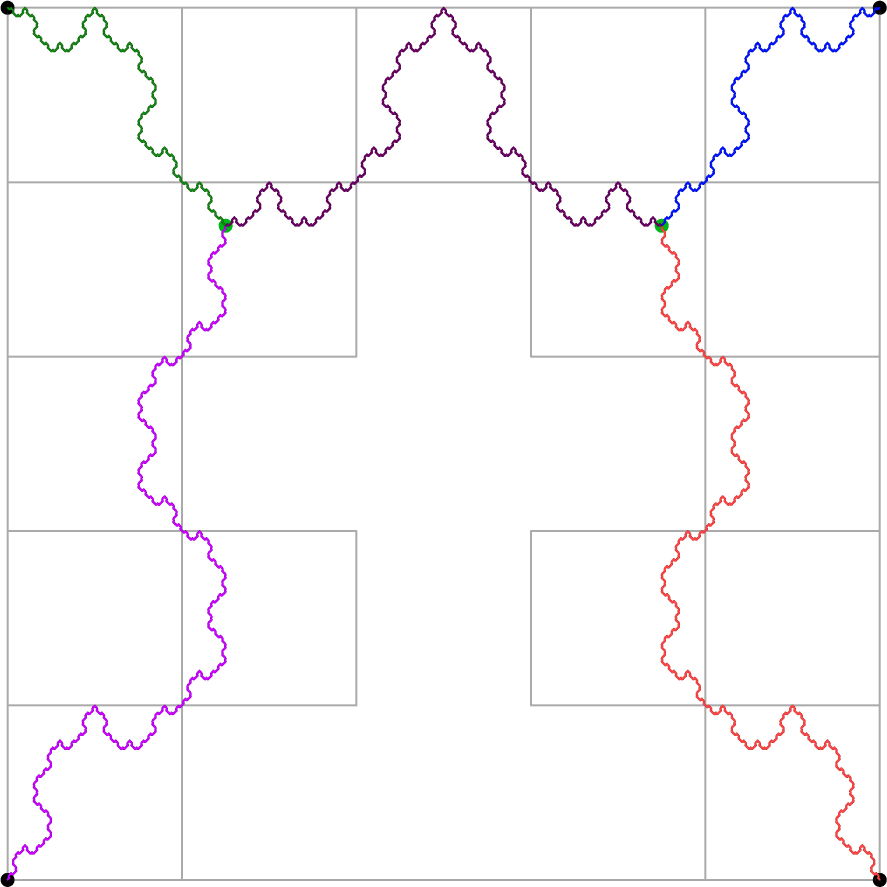}}
    \vspace{0.3cm}\vfill
    \fbox{
    \includegraphics[width=0.22\textwidth]{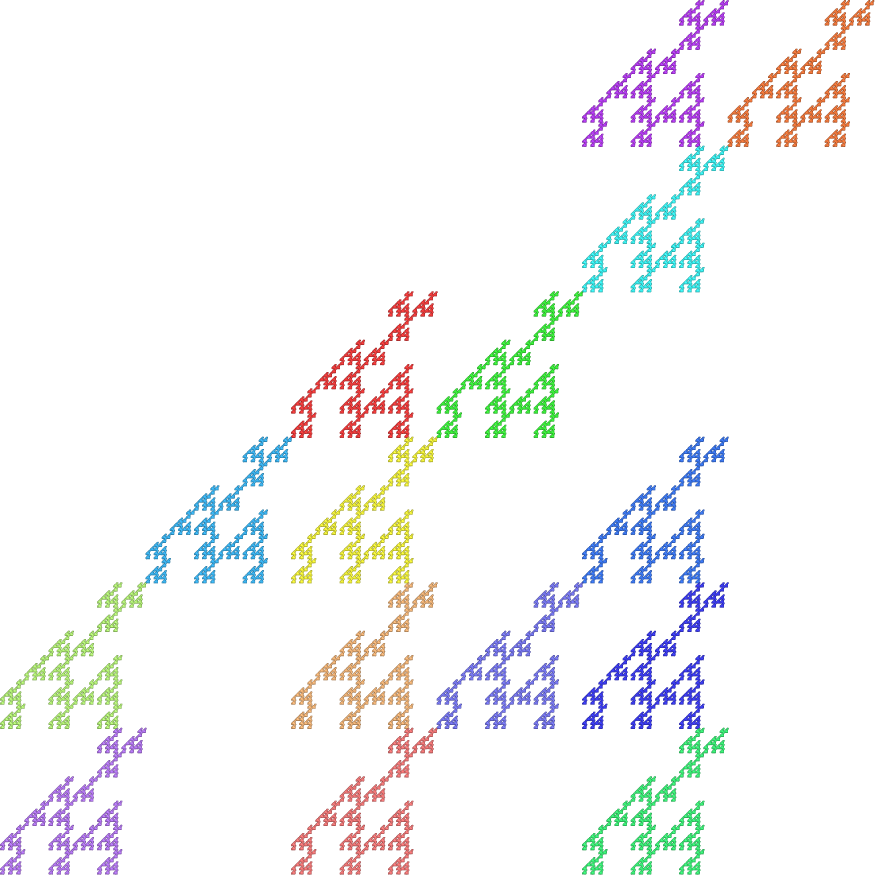}
    \includegraphics[width=0.22\textwidth]{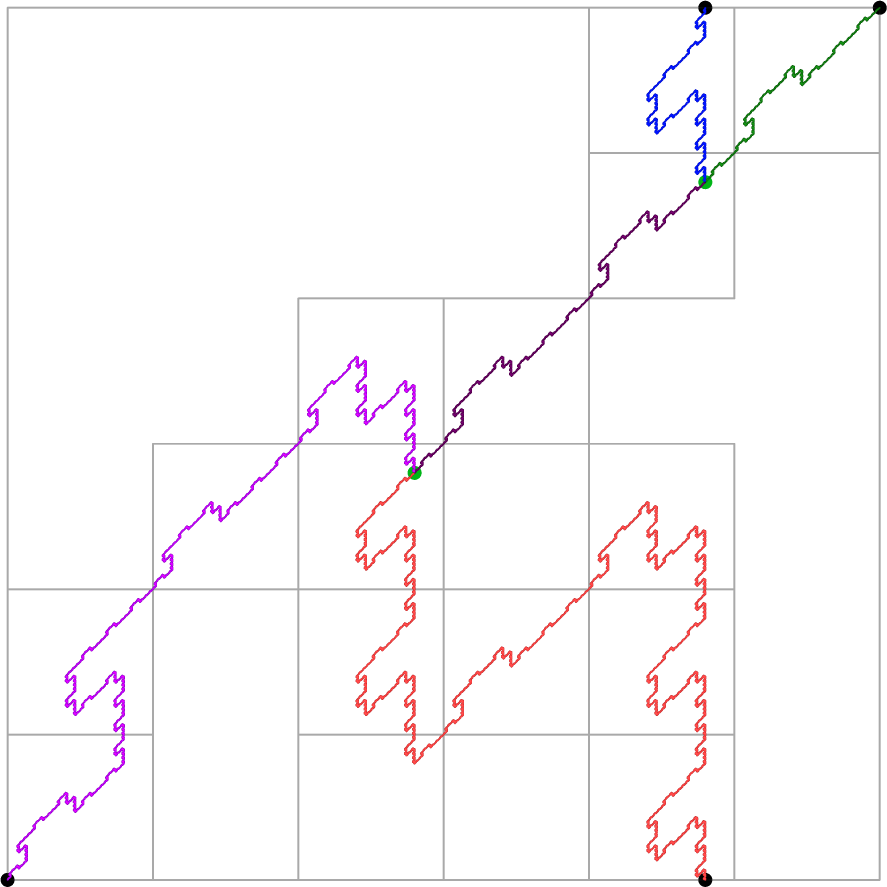}}
    \hfill
    \fbox{
    \includegraphics[width=0.22\textwidth]{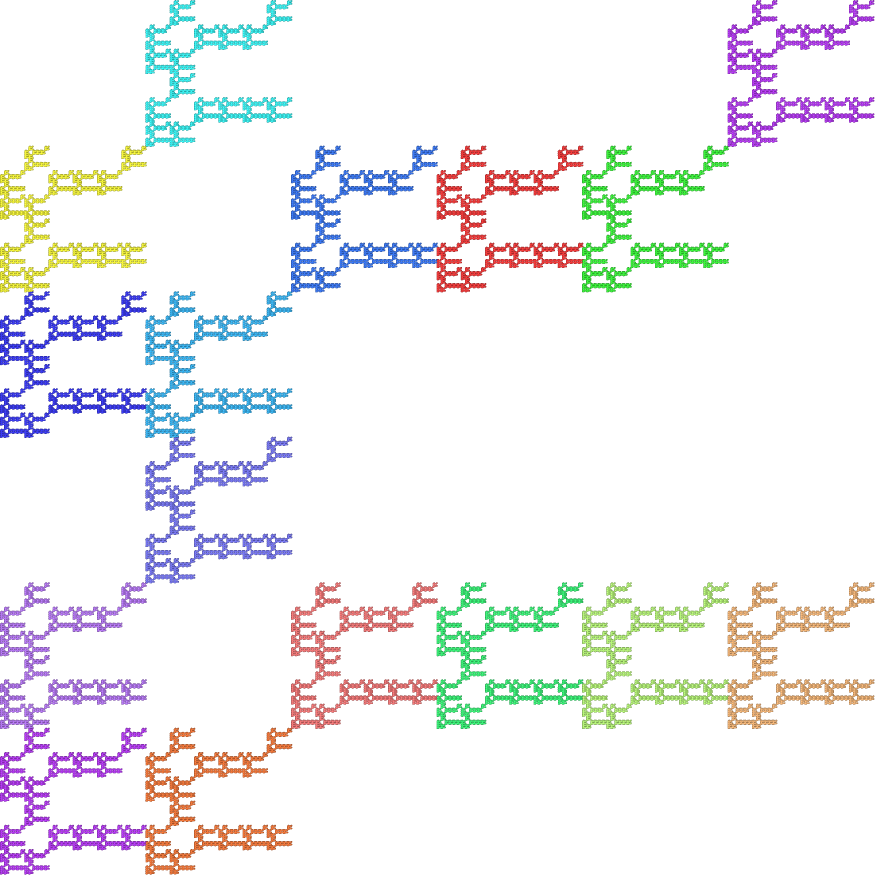}
    \includegraphics[width=0.22\textwidth]{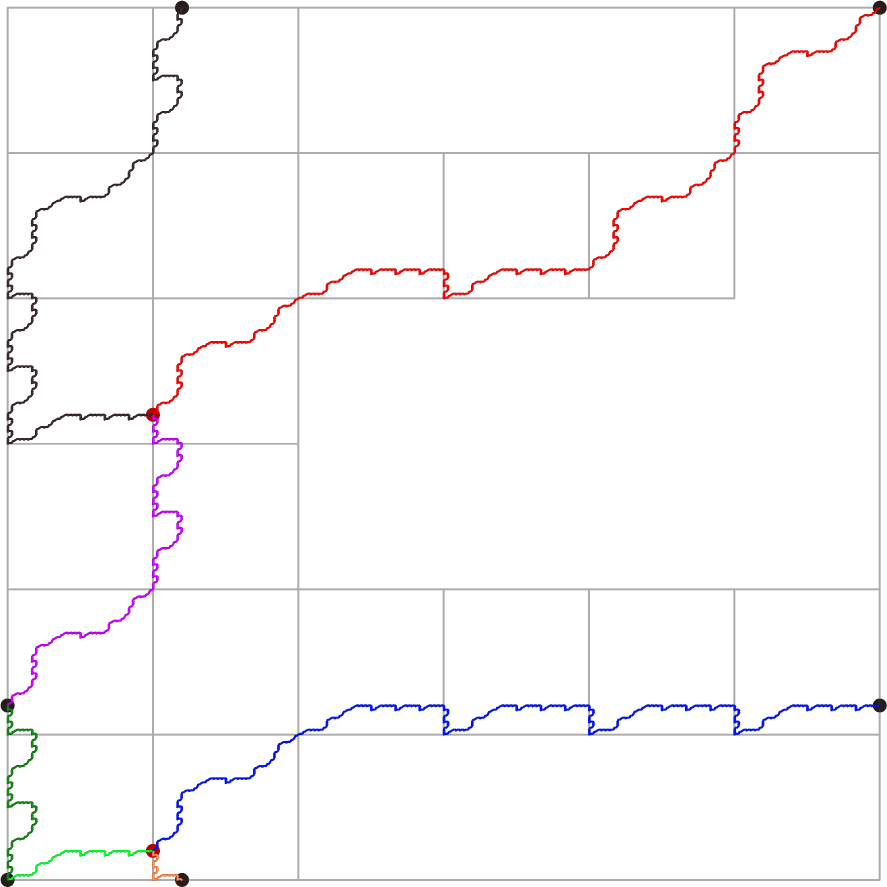}}
    \caption{Classes 10 ({\bf 4-A}), 11 ({\bf 4-B}), 12 ({\bf 4-C}) and 13 ({\bf 4-D6})}
    \label{fig:tree4}
\end{figure}

\begin{figure}[H]
    \centering
    \includegraphics[scale=1.1]{mt5.pdf}
    \hfill
    \fbox{
    \includegraphics[width=0.22\textwidth]{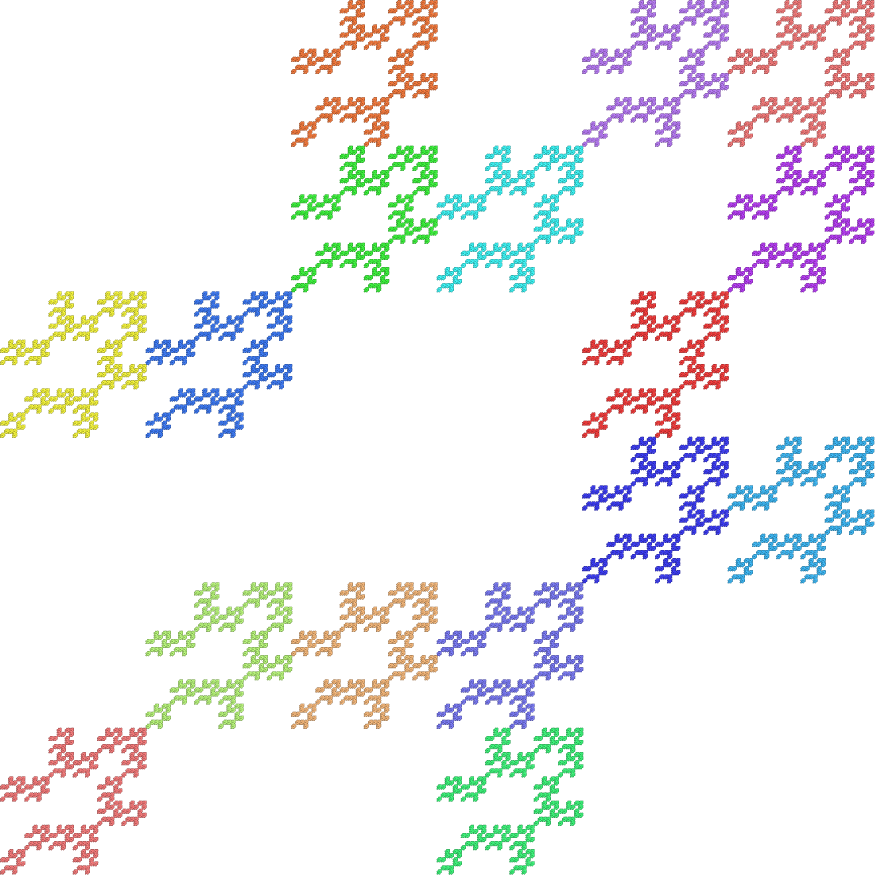}
    \includegraphics[width=0.22\textwidth]{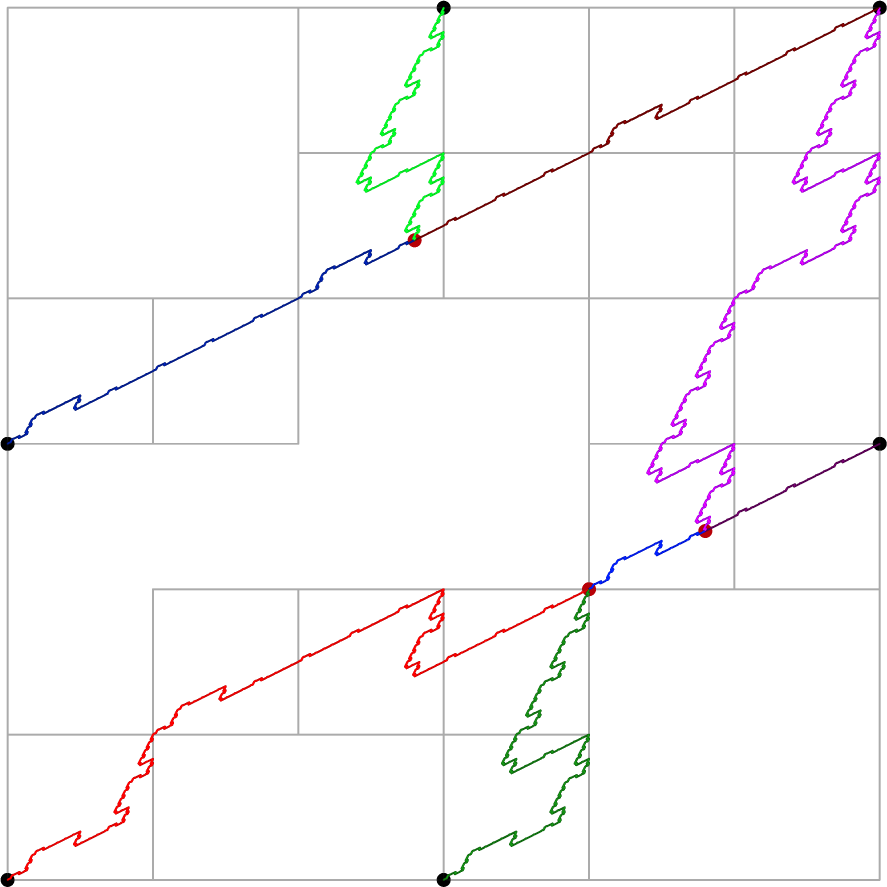}}
    \caption{Class 14 ({\bf 5-D6})}
    \label{fig:tree5}
\end{figure}

\begin{figure}[H]
    \centering
    \includegraphics[scale=1.1]{mt7.pdf}
    \hfill
    \fbox{
    \includegraphics[width=0.22\textwidth]{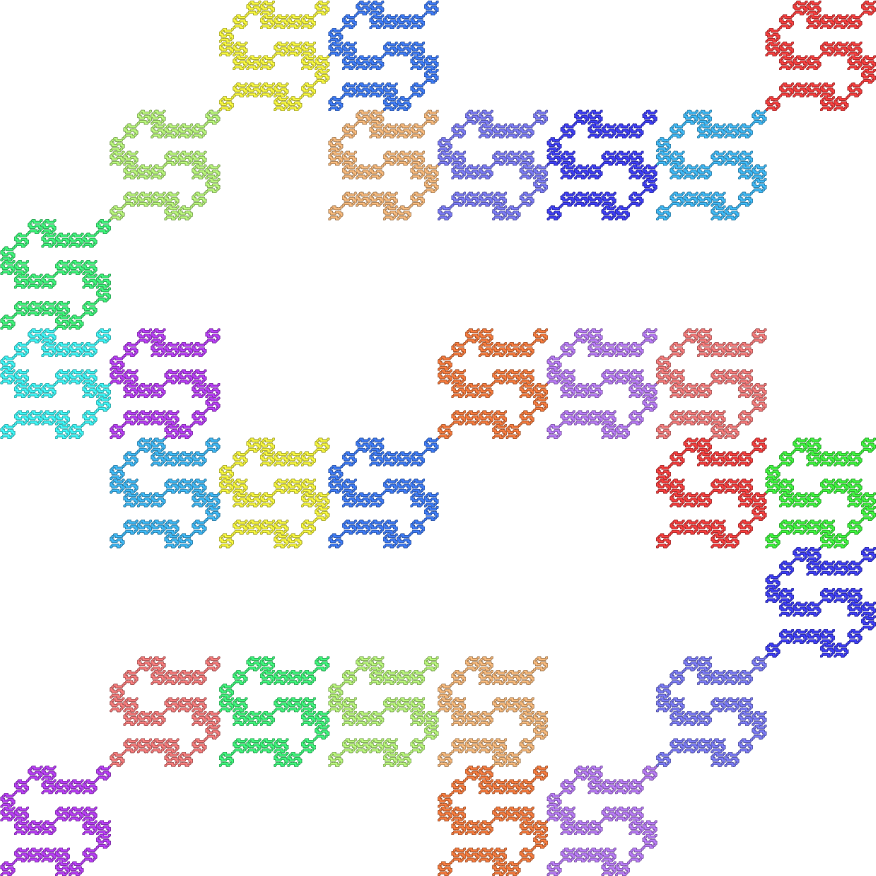}
    \includegraphics[width=0.22\textwidth]{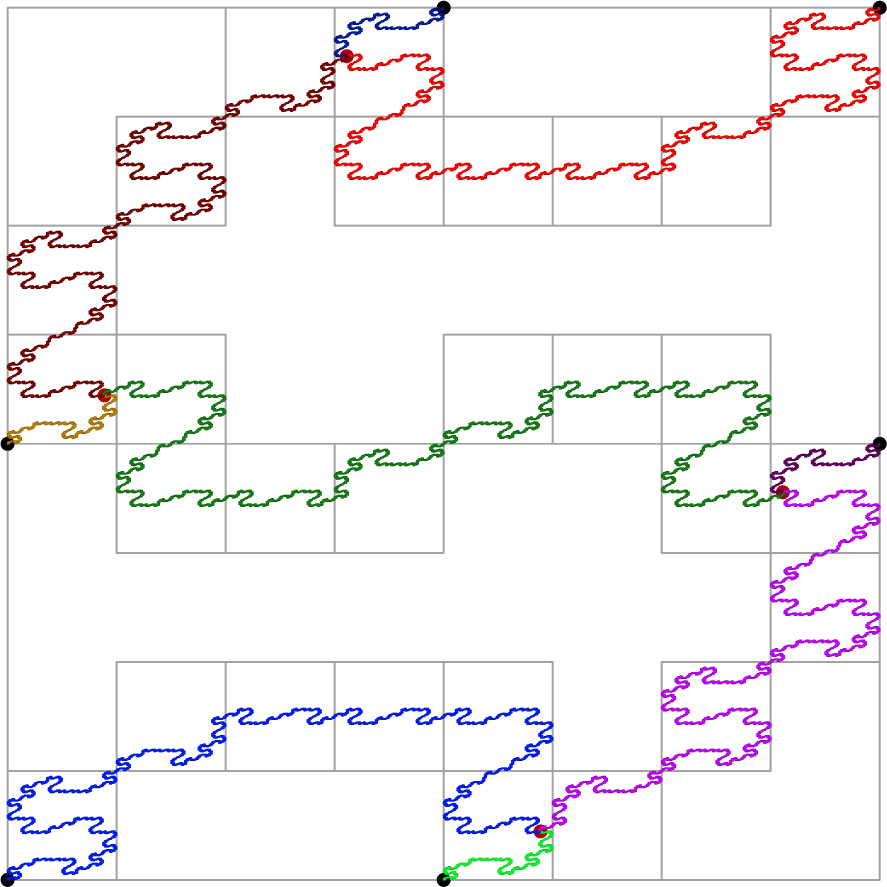}}
    \caption{Class 15 ({\bf 6-D6})}
    \label{fig:tree6}
\end{figure}

\begin{figure}[H]
    \centering
    \includegraphics[scale=1.1]{mt6.pdf}
    \hfill
    \fbox{
    \includegraphics[width=0.22\textwidth]{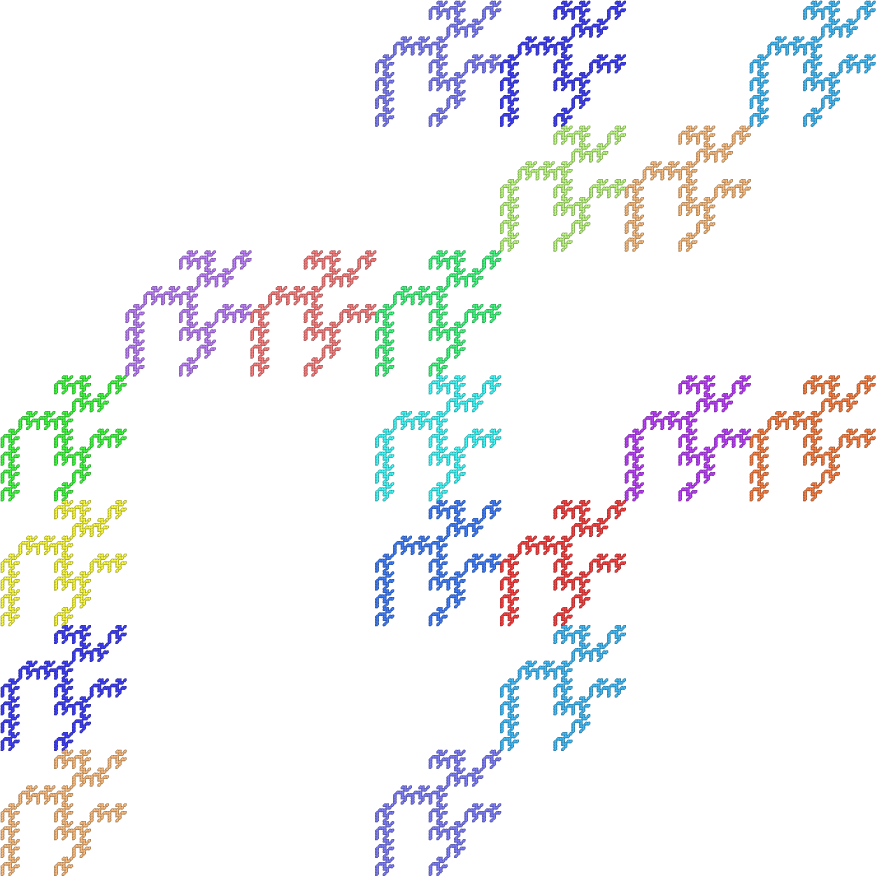}
    \includegraphics[width=0.22\textwidth]{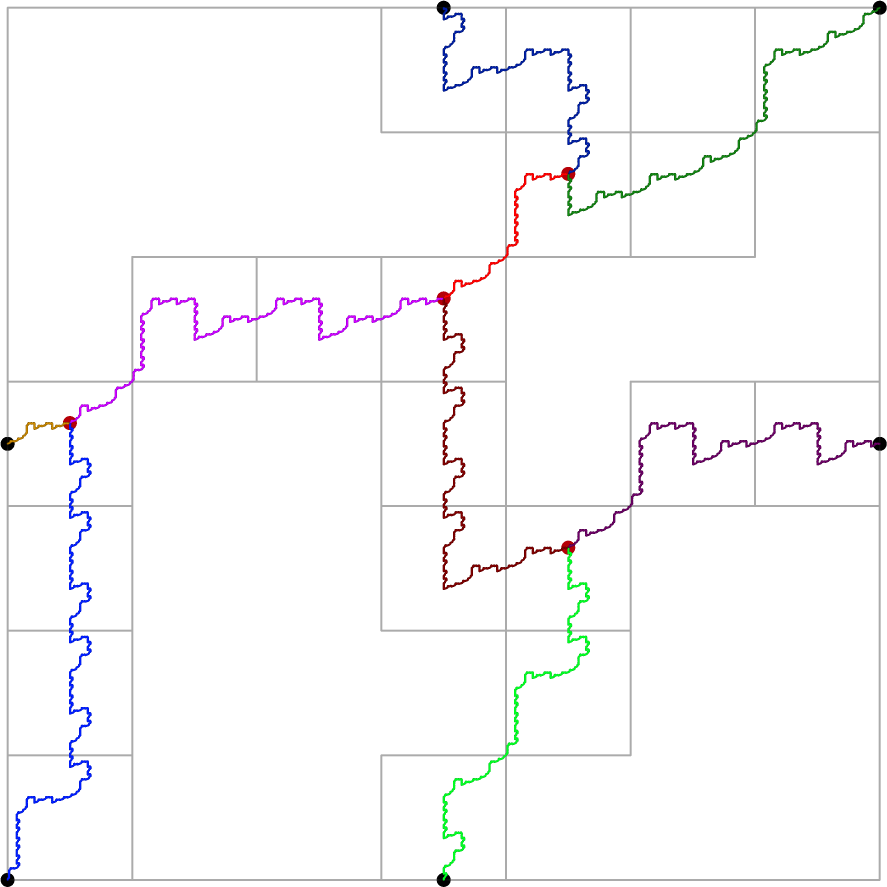}}
    \caption{Class 16 ({\bf 7-D6})}
    \label{fig:tree7}
\end{figure}

{\bf Acknowledgements:} The authors thank the reviewers for their careful reading, constructive comments, and suggestions. The authors are especially thankful to L.~L.~Cristea for her idea of labyrinth fractals and to H.~Rao whose question eventually motivated this research.\\

{\bf  Funding information:} The work is supported by the Mathematical Center in Akademgorodok under the
agreement no.075-15-2022-281 with the Ministry of Science and Higher Education of the Russian Federation.\\

{\bf Author contributions: }All authors have accepted responsibility for the entire content of this manuscript and approved its submission.\\

{\bf Conflict of interest:}  Authors state no conflict of interests.

\end{document}